\documentclass{amsart}
\pdfoutput=1
\usepackage[utf8]{inputenc} 

\usepackage[T1]{fontenc}    

\usepackage{url}

\usepackage{amsmath}
\usepackage{amsfonts}
\usepackage{amssymb}
\usepackage{amsthm} %
\usepackage{mathrsfs} %
\usepackage{enumerate} 
\usepackage{stmaryrd}
\usepackage{dirtytalk}

\usepackage{tikz}
\usepackage{pgfplots}
\usetikzlibrary{calc}
\usetikzlibrary{shapes}

\usetikzlibrary{intersections}
\usetikzlibrary{patterns}
\usepackage{comment}
\usepackage[all]{xy}

\usepackage{graphicx}
\usepackage{caption}
\definecolor{byzantine}{rgb}{0.74, 0.2, 0.64}
\definecolor{magenta}{rgb}{1.0, 0.0, 1.0}
\definecolor{islamicgreen}{rgb}{0.0, 0.56, 0.0}
\definecolor{ferrarired}{rgb}{1.0, 0.11, 0.0}

\definecolor{crimson}{rgb}{0.86, 0.08, 0.24}
\definecolor{applegreen}{rgb}{0.55, 0.71, 0.0}
\definecolor{ao}{rgb}{0.0, 0.5, 0.0}

\usepackage[hyperfootnotes,colorlinks=true,citecolor=cyan,backref=page]{hyperref}

\theoremstyle{plain} 
\newtheorem{thm}{Theorem}[section]
\newtheorem{cor}[thm]{Corollary}
\newtheorem{prop}[thm]{Proposition}
\newtheorem{lem}[thm]{Lemma}
\newtheorem{conj}{Conjecture}

\theoremstyle{definition}
\newtheorem{defi}[thm]{Definition}

\newtheorem{remark}[thm]{Remark}

\newtheorem{ex}[thm]{Example}

\newcommand{\cone}{\operatorname{cone}}
\newcommand{\shi}{{\operatorname{Shi}}}
\newcommand{\cat}{{\operatorname{Cat}}}
\newcommand{\aff}{{\operatorname{Aff}}}

\author[N. Chapelier-Laget]{Nathan~Chapelier-Laget}
\address[Nathan Chapelier-Laget]{CNRS, Universit\'e de Tours\\
Institut Denis Poisson \\ Facult\'e des sciences et techniques, Parc de Grandmont\\
37200 Tours\\ France}
\email{nathan.chapelier@gmail.com}
\urladdr{https://www.nathanchapelier.fr/home}

\author[C. Hohlweg]{Christophe~Hohlweg}
\address[Christophe Hohlweg]{Universit\'e du Qu\'ebec \`a Montr\'eal\\
LaCIM et D\'epartement de Math\'ematiques\\ CP 8888 Succ. Centre-Ville\\
Montr\'eal, Qu\'ebec, H3C 3P8\\ Canada}
\email{hohlweg.christophe@uqam.ca}
\urladdr{http://hohlweg.math.uqam.ca}

\title{Shi arrangements and low elements in affine Coxeter groups}

\thanks{This work was supported by the NSERC  grant {\em algebraic and geometric combinatorics of Coxeter groups} held by the second author.}

\begin{document}
\date{\today}
\maketitle

\begin{abstract} Given an affine Coxeter group $W$, the corresponding Shi arrangement is a refinement of the corresponding Coxeter hyperplane arrangements that was introduced by Shi to study Kazhdan-Lusztig cells for $W$. Shi showed that each region of the Shi arrangement contains exactly one element of minimal length in $W$.  Low elements in $W$ were introduced to study the word problem of the corresponding Artin-Tits (braid) group and turns out to produce automata to study the combinatorics of reduced words in $W$. 

In this article we show, in the case of an affine Coxeter group, that the set of minimal length elements of the regions in the Shi arrangement is precisely the  set of low elements, settling a conjecture of Dyer and the second author in this case. As a byproduct of our proof, we show that the descent-walls -- the walls that separate a region from the fundamental alcove --  of any region in the Shi arrangement are precisely the descent walls of the alcove of its corresponding low element.
\end{abstract}

\setcounter{tocdepth}{1}
\tableofcontents

\section{Introduction}  Let $(W,S)$ be a Coxeter system with length function $\ell:W\to \mathbb N$.  Let $V$ be the $ \mathbb{R}$-vector space with basis the simple system $\Delta=\{\alpha_s~|~s \in S\}$.  Let $B$ be the symmetric bilinear form on $V$ defined by:
$$
 B(\alpha_s,\alpha_t)=  \left\{
                          	\begin{array}{ll}
 						  -\text{cos}(\frac{\pi}{m_{st}})  & \text{if}~~ m_{st} < \infty \\
  				      	~~~~-1      & \text{if}~~m_{st} = \infty.
					    \end{array}
					    \right.
$$
Denote by $O_B(V)$ the orthogonal group for the pair $(V,B)$. For each $s \in S$ we consider the reflection $ \sigma_s : V \rightarrow V$ by $\sigma_s(x) = x - 2B(\alpha_s,x)\alpha_s$. The map $\sigma :W \hookrightarrow O_B(V)$ defined by $s \mapsto \sigma_s$ is a \emph{geometrical representation} of $(W,S)$.  The orbit $\Phi=W(\Delta)$ is a root system of $(W,S)$, with  positive root system $\Phi^+=\cone(\Delta)\cap \Phi$, where $\cone(X)$ is the set of nonnegative  linear combinations of vectors in $X$. The {\em inversion set of $w\in W$} is the set 
$$
N(w)=\Phi^+\cap w(\Phi^-),
$$ 
where $\Phi^-=-\Phi^+$. It is well-known that $\ell(w)=|N(w)|$. 

\subsection{Tits cone and Coxeter arrangement}

The \emph{contragredient representation} $\sigma^*:W\to \text{GL}(V^*)$ of $\sigma$, is defined as follows: for $w \in W$ and $f \in V^*$ we have 
$wf := \sigma^*(w)(f) = f\circ \sigma(w^{-1}).$  For $\alpha \in \Phi$ we consider the following hyperplane $H_\alpha$ and  open half-space $H_\alpha^+$:
$$
H_{\alpha} := \{ f \in V^*~| ~f(\alpha)  = 0 \}~\text{~and~}~H_\alpha^+:= \{ f \in V^*~| ~f(\alpha)  > 0 \}.
$$
The intersection $C$ of all $H_\alpha^+$ for $\alpha \in \Delta$ is called the \emph{fundamental chamber}. Let $D =\overline{C}$ be the topological closure of $C$, the \emph{Tits cone} $ \mathcal{U}(W)$ of $W$ is the following cone in $V^*$: 
$$
~\mathcal{U}(W) := \bigcup\limits_{w \in W}wD.
$$ 

If $w \in W$, $wC$ is called a \emph{chamber} of $\mathcal{U}(W)$. The action of $W$ on $\{wC, w \in W \}$ is simply transitive;  the chambers of $\mathcal{U}(W)$  are therefore in bijection with the elements of $W$. The {\em fundamental chamber $C$} corresponds to the identity element $e$ of $W$.  The hyperplane arrangement $\mathcal A_{(W,S)}=\{H_{\alpha}~|~\alpha \in \Phi\}$ is called  the {\em Coxeter arrangement of $(W,S)$}; see \cite{Hu90} for more details.

\subsection{Small roots, Shi arrangement and minimal elements} The {\em Shi arrangement} of an affine Weyl group $(W,S)$ is an affine hyperplane arrangement introduced  by Shi in~\cite[Chapter 7]{Shi86} to study Kazhdan-Lusztig cells in affine Weyl groups of type $A$ and later extended to other affine types in~\cite{Shi88}; see also Fishel's survey~\cite{Fi19} for more information.

Shi arrangement are generalized to any Coxeter system as follows.
Introduced by Brink and Howlett~\cite{BrHo93} to prove the automaticity of Coxeter groups, the {\em dominance order}  is the partial order $\preceq_{\mathrm{dom}}$ on $\Phi^+$ defined by:
$$
\alpha \preceq_{\mathrm{dom}} \beta \Longleftrightarrow \forall w \in W, \beta \in N(w) \Longrightarrow \alpha \in N(w).
$$
We say that $\beta$ is a \emph{small root} if $\beta$ dominates no other root than itself. The set of small roots is denoted by $\Sigma(W,S)$. One of the remarkable results of Brink and Howlett in~\cite{BrHo93} is that the set $\Sigma(W,S)$ is finite (if $S$ is finite); we refer the reader to the book~\cite[Chapter 4]{BjBr05} for more details.

The \emph{Shi arrangement} of $W$ is the subarrangement of the Coxeter arangement $\mathcal A_{(W,S)}$ defined by 
$$
\shi(W,S) = \{ H_\alpha ~|~ \alpha\in \Sigma(W,S) \}.
$$ 
The \emph{Shi regions} of $W$ are the connected components of $\shi(W,S)$ in the Tits cone $\mathcal U(W)$, that is, the connected components of 
$$
\mathcal{U}(W) ~\backslash \bigcup_{\alpha \in \Sigma(W,S)} H_{\alpha}.
$$

Each Shi region is uniquely determined by a union of chambers in the Tits cone, leading to an equivalence relation $\sim_\shi$ on $W$ defined by:  $u \sim_\shi w$ if and only if~$uC$ and $wC$ are in the same Shi region.  Let us define
$$
L_\shi (W,S) := \{w\in W\mid u\sim_\shi w \implies \ell(w)\leq \ell(u) \}
$$

\subsection{Low elements}
A conjectural characterization of $L_\shi(W,S)$ is given in terms of {\em low elements}. Low elements were introduced by Dehornoy, Dyer and the second author~\cite{DDH14} in relation to Garside theory~\cite{Gars15} in the case of  Artin-Tits (braid) groups. 
Low elements are defined as follows: $w\in W$ is a {\em low element} if
$$
 N(w)=\cone(\Sigma(W,S)\cap N(w))\cap \Phi.
 $$ 
 
 Denote by $L(W,S)$ the set of low elements of $(W,S)$. Dyer and the second author show in \cite{DyHo16} that $L(W,S)$ is a finite {\em Garside shadow}, i.e., a subset of $W$ containing~$S$ and closed under taking suffixes and under taking joins in the right weak order. 
 
 The following conjecture is a reformulation  of~\cite[Conjecture 2]{DyHo16} in the case of low elements; see also \cite[\S3.6]{HoNaWi16}.

\begin{conj}\label{conj:Main} We have $L_\shi(W,S)=L(W,S)$. Moreover, any Shi region $\mathcal R$  contains a unique low element, which is the unique element of minimal length in~$\mathcal R$.  
\end{conj}

 Both statements in the above conjecture are equivalent, see \S\ref{ss:ShiSmall}. This conjecture implies also, if true, that the set $L_\shi(W,S)$ is a Garside shadow, since the set of low elements is a Garside shadow, as shown in~\cite{DyHo16}.  In~\cite{HoNaWi16}, the authors show that any Garside shadow is the set of states of an automaton that recognizes the language of reduced words for $(W,S)$. They also show that  the Brink-Howlett automaton, whose states are indexed by $L_\shi(W,S)$, projects onto the Garside shadow automaton associated to $L(W,S)$. 
 
 Therefore Conjecture~\ref{conj:Main} is known to be true in the cases where the Brink-Howlett automaton is minimal. For instance, this conjecture is true for Coxeter systems with complete Coxeter graphs and right-angled Coxeter groups, see \cite[Theorem 1.3]{HoNaWi16}; all cases of a minimal Brink-Howlett automaton are characterized by Parkinson and Yau in~\cite{PaYa19}. However, the  sole affine Weyl groups  for which the Brink-Howlett automaton is minimal are of type $A$.  Recently, Charles~\cite{Charles21} proved Conjecture~\ref{conj:Main} for any rank~3 Coxeter system. 

\subsection{Main results} In this article we prove Conjecture~\ref{conj:Main} for any affine Coxeter systems.   
 
\begin{thm}\label{thm:main} If $(W,S)$ is an affine Coxeter system, then $L_\shi(W,S)=L(W,S)$. 
\end{thm}

\begin{cor} If $(W,S)$ is an affine Coxeter system, then $L_\shi(W,S)$ is a Garside shadow for $(W,S)$. 
\end{cor}

The above corollary shows that  the Brink-Howlett automaton is the Garside shadow automaton, in the sense of~\cite{HoNaWi16},  associated to $L(W,S)$.

\smallskip
In the case of an (irreducible) affine Coxeter system $(W,S)$, Shi shows in \cite[Proposition 7.2]{Shi88} that each equivalence class under $\sim_\shi$ contains exactly one element of minimal length, which means,  by Theorem~\ref{thm:main}, that $|L(W,S)|=|L_\shi(W,S)|$ counts the number of regions of $\shi(W,S)$. Writing $(W_0,S_0)$ for the underlying finite Coxeter system of $(W,S)$, Shi provides the following remarkable formula in~\cite[Theorem 8.1]{Shi88}:
 $$
 |L_\shi(W,S)|= (h+1)^{n},
 $$
 where $n=|S_0|$ is the rank of $W_0$ and $h$ is the Coxeter number of $W_0$. This formula also appears in relation to diagonal invariants in the work of Haiman~\cite[\S7]{Ha94}: Haiman conjectured that there exist a $W_0$-stable quotient ring of the covariant ring $\mathbb C[V\oplus V^*]^{co\, W}$ of dimension $(h+1)^{n}$; this conjecture was confirmed by Gordon in~\cite{Go03}. As a corollary of~Theorem~\ref{thm:main} we obtain the enumeration of the set of low elements in the case of affine Coxeter systems.
  
 \begin{cor} If $(W,S)$ is an affine Coxeter system with underlying finite Weyl group $W_0$ of rank $n$ and Coxeter number $h$,  then the number of low elements in $(W,S)$ is $|L(W,S)| = (h+1)^{n}$.
\end{cor}

\subsection{Outline of the article} The proof of Theorem~\ref{thm:main} can be easily reduced to the case of irreducible affine Coxeter systems by standard technics; see for instance \cite{DyLe11}. Therefore, all Coxeter systems considered in this article are irreducible.

In order to prove Theorem~\ref{thm:main}, we  first recollect in~\S\ref{se:Affine} well-known facts about affine Coxeter systems and affine Weyl groups. In~\S\ref{se:Low}, we survey the notions of  low elements.  In \S\ref{se:Shi}, we survey Shi arrangements and  Shi's admissible signs.  In particular, in \S\ref{ss:LowShi} we show that $L_\shi(W,S)$ is closed under taking suffixes (Proposition~\ref{prop:ShiS}), and in \S\ref{ss:ShiSmall} we reduce the proof of Theorem~\ref{thm:main} to showing that $L_\shi(W,S)\subseteq L(W,S)$. 

Then we tackle the core of the proof of Theorem~\ref{thm:main} in \S\ref{se:Signs} to \S\ref{se:Proof}. Firstly,  we characterize the so-called {\em descent-walls} of a Shi region $\mathcal R$, that is,  the walls that separate $\mathcal R$ from the fundamental alcove (see Definition \ref{descent wall/root}); this characterization is given in terms of Shi's admissible sign types and Shi coefficients (Proposition~\ref{cor:Key1}). This leads to the first key result, which is Lemma~\ref{lem:Key1}, that describes some descent-walls of a Shi region. 

 Secondly, we observe that Cellini and Papi's set of minimal elements, which  appears in their works~\cite{CePa00,CePa02} on $ad$-nilpotent ideals of a Borel subalgebra, is precisely the set $L^0(W,S)$ of low elements in the dominant region of $\mathcal A(W,S)$, i.e., the fundamental region of the Coxeter arrangement of $(W_0,S_0)$ (Corollary~\ref{cor:LLshiDom}). This cover the initial case of our proof by induction. In order to conclude our inductive proof, we prove our second key result, the {\em descent-wall theorem}: the descent-walls of $w\in L_\shi$ are precisely those of its corresponding Shi region (Theorem~\ref{thm:ShiLDes}). We conclude the proof of Theorem~\ref{thm:main} in~\S\ref{ss:ProofMain}.

The connection with Cellini and Papi's results implies in particular the following corollary where the second equality is already known via the last term,  understood \cite[Theorem 1]{CePa02} as the number of ad-nilpotent ideals of a Borel subalgebra $\mathfrak{b}$ of a simple Lie algebra $\mathfrak{g}$.

\begin{cor} If $(W,S)$ is an irreducible affine Coxeter system with underlying finite Weyl group $W_0$, then the number of low elements in the fundamental chamber of $\mathcal A(W,S)$ is the $W_0$-Catalan number:
$$
|L^0(W,S)|=\cat(W_0) = \frac{1}{|W_0|}\prod_{i=1}^n (h+e_i+1),
$$
where $e_1,\dots,e_n$ are the exponents of $W_0$.
\end{cor}
For a survey of the history of Coxeter-Catalan numbers, we refer the reader to~\cite[p.181]{At05}. For more details on Coxeter numbers and exponents see \cite[Chapter 3]{Hu90}. 

%
%

\section{Affine Coxeter systems}\label{se:Affine}

For a Euclidean vector space $V_0$ with inner product $\langle\cdot,\cdot\rangle$,  we denote by $O(V_0)$ the group of isometries on $V_0$ and $\aff(V_0)$ the group of affine isometries on~$V_0$.  The aim of this section is to survey some well-known facts on affine Coxeter systems along the lines of~\cite[\S3]{DyLe11}; see also \cite[Chapters 6 \& 7]{Kac90}.

\subsection{Weyl group} \label{ss:CRS}

Let $\Phi_0$ be an irreducible crystallographic root system in $V_0$ with simple system $\Delta_0$; set $n=|\Delta|$. Let $\Phi_0^+$ be the associated positive root system. The root poset on $\Phi_0^+$ is the poset $(\Phi_0^+,\preceq)$ defined by $\alpha\preceq \beta$ if and only if $\beta-\alpha \in\cone(\Delta_0)$.

For any $\alpha \in \Phi_0$ let $\alpha^\vee=\frac{2}{\langle\alpha, \alpha \rangle}\alpha$ be the coroot of $\alpha$.  The reflection $s_\alpha$ associated to $\alpha$ is defined by
$$
\begin{array}{ccccc}
s_{\alpha}  & : & V_0 & \longrightarrow & V_0 \\
                 &   & x& \longmapsto     & x-\langle \alpha, x \rangle\alpha^\vee.
\end{array}
$$

Notice that for all $\alpha \in \Phi_0$ we have 
\begin{align}\label{2}
\langle \alpha^{\vee}, \alpha \rangle = \langle \frac{2\alpha}{\langle\alpha, \alpha \rangle}, \alpha \rangle = 2.
\end{align}

The set of fixed points of $s_\alpha$ is the hyperplane   $H_\alpha:=\{x\in V_0 \mid \langle x,\alpha\rangle = 0\}$. Let $W_0$ be the Weyl group associated to $\Phi_0$. We identify the root lattice $Q:= \mathbb{Z}\Phi_0$ and the coroot lattice $Q^{\vee} := \mathbb{Z}\Phi_0^{\vee}$ with their group associated translations and we denote by $\tau_x$ the translation corresponding to $x \in Q^{\vee}$.  With $S_0 = \{s_{\alpha}\mid \alpha\in\Delta_0\}$,  the pair $(W_0,S_0)$ is a finite Coxeter system.  The set of reflections  $T_0$ of $W_0$ is:
 $$
 T_0 = \bigcup\limits_{w \in W_0}wS_0w^{-1}.
 $$
The \emph{fundamental chamber} of $W_0$, also called the \emph{dominant region}, is  the set: $$C_\circ := \{x \in V_0~|~\langle x,\alpha \rangle > 0, ~\forall \alpha \in \Delta_0\}.$$

\subsection{Affine Weyl group and affine Coxeter arrangement}\label{affine cox arrangement} 

The affine reflection $s_{\alpha,k} \in \text{Aff}(V_0)$ is defined by:
\begin{align}\label{shi notation reflection}
 s_{\alpha,k}(x) &=x-(\langle \alpha, x \rangle-k)\alpha^{\vee} = \tau_{k\alpha^{\vee}}s_{\alpha}.
\end{align}

The {\em affine Weyl group} $W$  associated to $\Phi_0$ is the following discrete reflection group of affine isometries: 
$$
W = \langle s_{\alpha,k}~|~\alpha \in \Phi_0, ~k \in \mathbb{Z \rangle} = Q^{\vee} \rtimes W_0.
$$

Let $\alpha_0$ be the highest root of $\Phi_0$. The set $ S := S_0 \cup \{s_{\alpha_0,1}\}$  is a set of simple reflections for $W$. The set of fixed points of the reflection $s_{\alpha,k}$ is the affine hyperplane 
\begin{align*}
H_{\alpha,k}  = \{ x \in V~|~ \langle x, \alpha \rangle = k\}.
\end{align*}
The collection of hyperplanes $H_{\alpha,k}$,  for $ \alpha \in \Phi_0$ and $k \in \mathbb{Z}$,  is denoted by $\mathcal{A}(W,S)$ and is called the \emph{affine Coxeter arrangement} of $(W,S)$.  An \emph{alcove} of $(W,S)$ in $V_0$ is a connected component of  
$$
V_0 ~\backslash \bigcup\limits_{H \in \mathcal{A}(W,S)} H.
$$ 
The group $W$ acts transitively on the set of alcoves, i.e., the map $w\mapsto w\cdot A_\circ$ is a bijection between $W$ and the set of alcoves of $(W,S)$ in $V_0$ where $A_\circ$ is the alcove associated to the identity element.

\begin{remark} The affine Coxeter arrangement is in fact the intersection of the usual Coxeter arrangement with an affine hyperplane in the Tits cone. The fundamental alcove is the intersection of the fundamental chamber with this affine hyperplane. For more details, see  \cite[\S6.5]{Hu90}, \cite[Proof of Proposition 3]{DyLe11} or \cite[\S1.3.5]{Ch21}.
\end{remark}

\subsection{Shi coefficients and Shi parameterization} \label{para coef k}

In \cite[Proposition 5.1]{Shi87}, Shi describes a correspondence 
$
w \mapsto (k(w,\alpha))_{\alpha\in\Phi_0^+}
$
between $W$ (or equivalently  the set of alcoves of $(W,S)$) and some $\Phi_0^+$-tuples over $\mathbb Z$.   We refer to some of Shi's results from~\cite{Shi87,Shi88}. However, there are three differences in the conventions between Shi's articles and our article:
\begin{enumerate}
\item In Shi's articles, $W$ is denoted by $W_a$ and the underlying Weyl group is denoted by $W$.  

\item In Shi's articles, the root system of the affine Coxeter system $(W,S)$ is associated to $\Phi_0^\vee$ instead of $\Phi_0$ in this article. Since $(\Phi_0^\vee)^\vee = \Phi_0$, it is enough to apply the involution $\alpha\mapsto \alpha^\vee$ when applying Shi's results. 

\item The action of $W$ on the root system is from the right in Shi's articles, where we adopt the convention of acting by the left; for more details see \cite[Remark 2.1]{Ch20}. 
\end{enumerate}

By convention, we have the following formula~ \cite[\S1]{Shi87} relating Shi coefficients on negative roots and Shi coefficients on positive roots:
\begin{align}\label{neg shi coeff}
k(w,-\alpha) = -k(w,\alpha) ~\mathrm{~~for ~~any~} \alpha \in \Phi_0.
\end{align}

Now the following statements give very important relations between Shi coefficients on positive roots.  

\begin{lem}[{\cite[Lemma 3.1]{Shi87}}]  \label{k groupe fini}
Let $w$ be an element of $W_0$ and $\alpha \in \Phi^+$. Then 
\begin{align*}
k(w,\alpha) =  \left\{
                          	\begin{array}{rl}
 						0      & \text{if}~~ \alpha \notin N(w) \\
  				      	-1      & \text{if}~~\alpha \in N(w).
					    \end{array}
					    \right.
\end{align*}
\end{lem}



The following proposition is in general only stated for positive roots, but turns out to stand also for any root thanks to Eq.~(\ref{neg shi coeff}).

\begin{prop}\label{descente indice k}
Let $w \in W$,  $s \in S_0$ and $t \in T_0$.  For all $\alpha \in \Phi_0$ one has 
\begin{itemize}
\item[(1)] $k(sw, \alpha) = k(w, s(\alpha)) + k(s, \alpha)$  \text{~~} {\cite[Proposition 4.2]{Shi87}}.
\item[(2)] $k(tw, \alpha) = k(w, t(\alpha)) + k(t, \alpha)$ \text{(reformulation of} \text{~~} {\cite[Proposition 3.2]{Ch20}}).
\end{itemize}
\end{prop}

\begin{lem}[{\cite[Proposition 4.3]{Shi87}}] \label{lemma k<0}
Let $w \in W$ and  $\beta \in \Delta_0$. If $s_{\beta} \in D_L(w)$ then $k(w,\beta) \leq -1$.
\end{lem}

\begin{remark}\label{rem:use}  \begin{enumerate}
\item Shi does not give a proof of Lemma~\ref{lemma k<0} in \cite{Shi87}; instead he refers the reader to \cite[Chapter 7]{Shi86}.  We point out that a direct proof of this result can be obtain using an analog of \cite[Proposition 2.16]{DyHo16} for ``$p$-inversions sets $N^p(w)$''.
\item It is very useful to notice that, for an arbitrary $w\in W$ and $\alpha\in \Phi^+_0$, $|k(w,\alpha)|$ is given by the number of hyperplanes parallel to $H_{\alpha}$  that separates $w\cdot A_\circ$ from $A_\circ$; the sign is then negative if $H_{\alpha}$ is one of these hyperplanes.
\end{enumerate}
\end{remark}

\begin{ex}[Type $\tilde A_2$]\label{ex:A2} Consider $(W,S)$ of type $\tilde A_2$. Alcoves and their labelling are shown in Figure~\ref{fig:A2}. Write $S_0=\{s_1,s_2\}$, so $\Delta_0=\{\alpha_1,\alpha_2\}$ and $\Phi_0^+=\{\alpha_1,\alpha_2,\alpha_1+\alpha_2\}$ (with $s_i=s_{\alpha_i}$).

\begin{figure}[h!]
\includegraphics[scale=0.45]{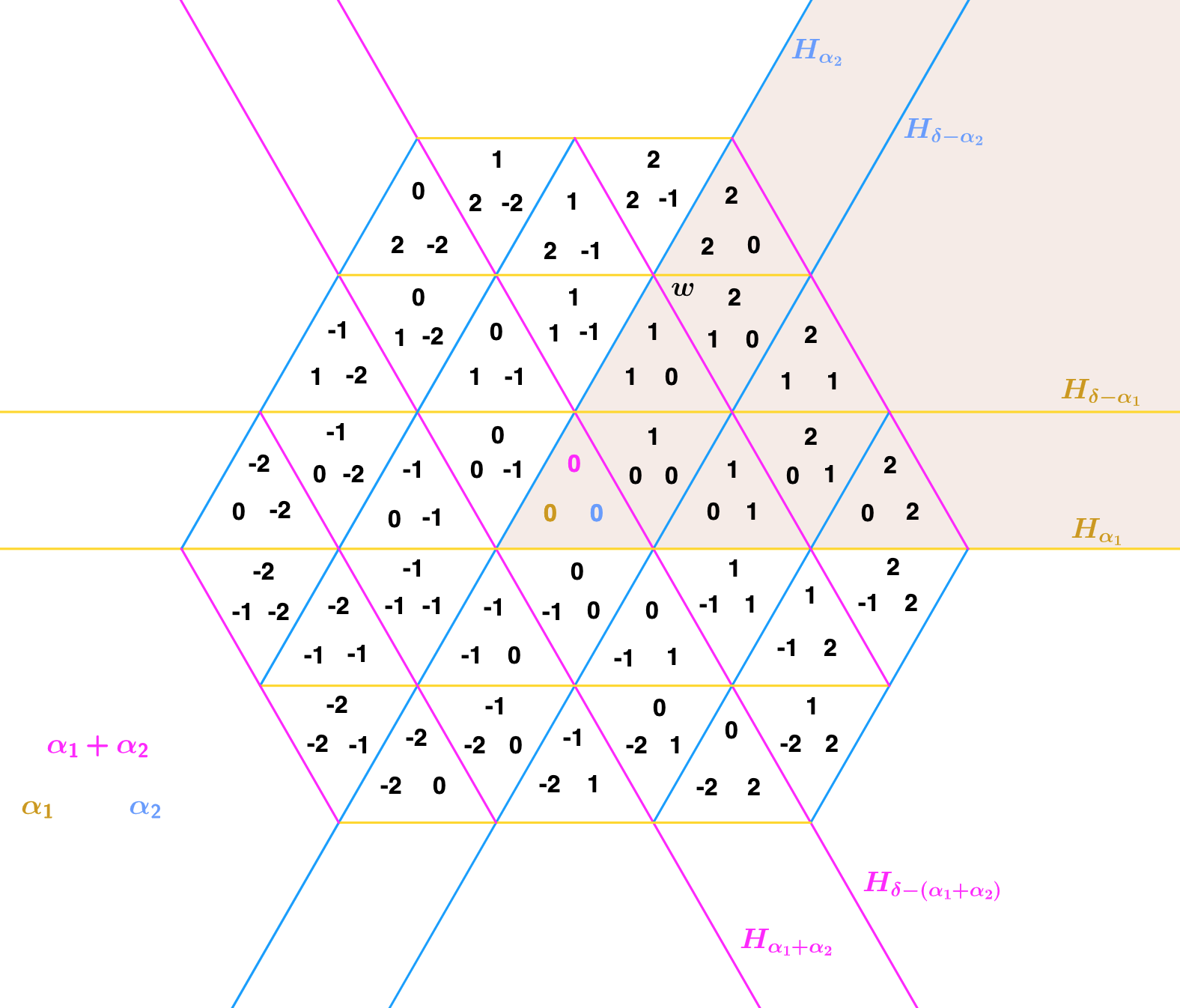}
\caption{The Coxeter arrangement of type $\tilde A_2$. Each alcove is labelled with its Shi parameterization. 
The labels $k(w,\alpha)$ for $\alpha\in \Phi_0^+$ are indicated in each alcove with the parameterization given at the left-hand side of the figure. The shaded region is the dominant region of $\mathcal A(W,S)$, which is also the fundamental chamber for the finite Weyl group $W_0$. The fundamental alcove $A_\circ$, which corresponds to $e$, is parameterized by $0$'s.}
\label{fig:A2}
\end{figure}
\end{ex}

\begin{ex}[Type $\tilde B_2$]\label{ex:B2} Consider $(W,S)$ of type $\tilde B_2$, with underlying finite Weyl group of type $B_2$.  We consider here $V=\mathbb R_2$ with orthonormal basis $\{e_1,e_2\}$. Set $\alpha_1=e_1$ and $\alpha_2=e_2-e_1$. Write $S_0=\{s_1,s_2\}$, so $\Delta_0=\{\alpha_1,\alpha_2\}$ and $\Phi_0^+=\{\alpha_1,\alpha_2,\alpha_1+\alpha_2,2\alpha_1+\alpha_2\}$ (with $s_i=s_{\alpha_i}$). The long roots are $\alpha_2$ and $2\alpha_1+\alpha_2$. Alcoves and their labelling are shown in Figure~\ref{fig:B2}.

\begin{figure}[h!]
\includegraphics[scale=2.8]{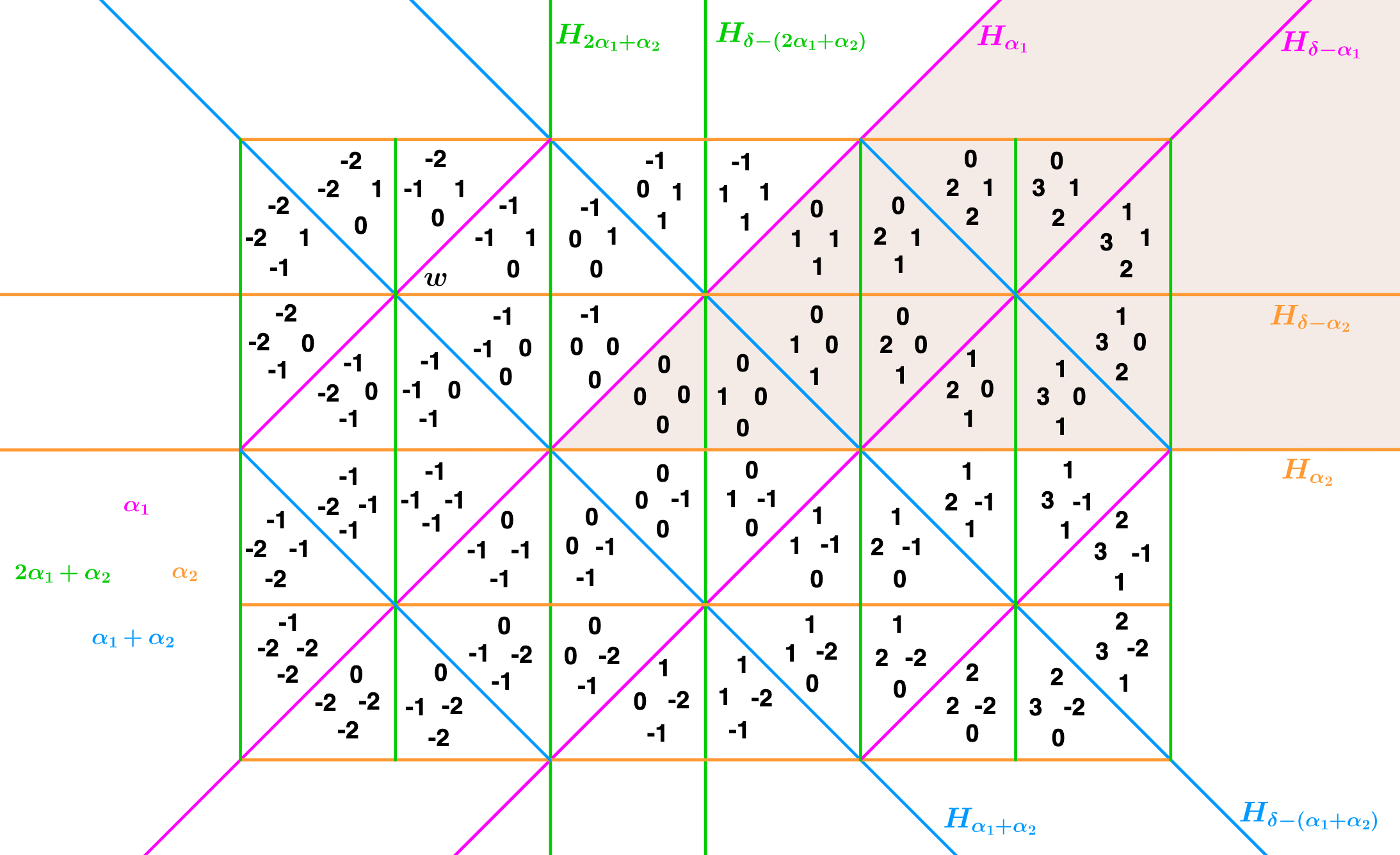}
\caption{The Coxeter arrangement of type $\tilde B_2$. Each alcove is labelled with its Shi parameterization. 
The labels $k(w,\alpha)$ for $\alpha\in \Phi_0^+$ are indicated in each alcove with the parameterization given at the left-hand side of the figure. The shaded region is the dominant region of $\mathcal A(W,S)$, which is also the fundamental chamber for the finite Weyl group $W_0$.  The fundamental alcove $A_\circ$, which corresponds to $e$, is parameterized by $0$'s.}
\label{fig:B2}
\end{figure}
\end{ex}

\subsection{Affine crystallographic root systems}\label{ss:Aff1}

Set $V=V_0\oplus \mathbb R \delta$ to be a real vector space with basis $\Delta_0\sqcup \{\delta\}$, where $\delta$ is an indeterminate.  We define a symmetric bilinear form $B:V\times V\to \mathbb R$ by extending $\langle \cdot,\cdot\rangle$ to $V$ as follows: for $\alpha,\beta \in \Delta_0$ we set
$
B(\alpha,\beta) = \langle \alpha,\beta\rangle,\  B(\alpha,\delta)=0 \textrm{ and } B(\delta,\delta)=0 .
$
The pair $(V,B)$ is a quadratic space for which  the isotropic cone  is  $\mathbb R\delta= \{x\in V\mid B(x,x)=0\}$. 

\smallskip
We  describe now a {\em geometric representation of $(W,S)$ on $(V,B)$}.  A simple system in $(V,B)$ for $(W,S)$ is
$
\Delta = \Delta_0\sqcup \{\delta - \alpha_0\}.
$
The  root system $\Phi$ and the positive root system $\Phi^+$ for $(W,S)$ in $(V,B)$ are defined by $\Phi=\Phi^+\sqcup \Phi^-$, where $\Phi^-=-\Phi^+$ and
$\Phi^+ =( \Phi_0^+ +  \mathbb N\delta ) \sqcup ( \Phi_0^-+ \mathbb N^*\delta)$. The pair $(\Phi,\Delta)$ is called a {\em (crystallographic) based root system of $(W,S)$ in~$(V,B)$}. 
For $\alpha+k\delta\in\Phi^+$, the {\em reflection} $s_{\alpha+k\delta}:V\to V$ is defined as follows. Let $x=x_0+a\delta\in V$, then: 
\begin{align}\label{eq:Ref2}
s_{\alpha+k\delta} (x):= x - 2\frac{B(\alpha+k\delta,x)}{B(\alpha+k\delta,\alpha+k\delta)}(\alpha+k\delta) =  s_\alpha(x_0) + (a - k\langle \alpha^\vee,x_0 \rangle)\delta .
\end{align}

Observe that $s_{\alpha+0\delta} = s_\alpha$.  It follows that
$
S=S_0\sqcup \{s_{\delta-\alpha_0}\}.
$
The correspondences  between the positive root system $\Phi^+$,  the reflections on $V$ and those on $V_0$ defined in \S \ref{affine cox arrangement}, are as follows:
$$
\begin{array}{ccccc}
\Phi^+ & \longrightarrow & \{ \text{reflections on} ~V \} & \longrightarrow & \{ \text{reflections on} ~V_0 \} \\
             \alpha+k\delta & \longmapsto & s_{\alpha +k\delta}  & \longmapsto & s_{-\alpha,k} .
\end{array}
$$
Therefore,  the hyperplane associated to $\alpha+k\delta \in \Phi^+$ is
$
 H_{\alpha+k\delta}:= H_{-\alpha,k}.
$

\section{Small roots and low elements  in affine Weyl groups}\label{se:Low}  Let $(W,S)$ be an affine Coxeter system with underlying Weyl group $W_0$. We now survey the notions of small roots and low elements, see \cite{DyHo16} for more details.  We gave general definitions of small roots and low elements in the introduction. In the case of affine Weyl groups these notions have easier interpretations, which we discuss below.

\subsection{Small roots and low elements} \label{small root} A positive root $\alpha+k\delta\in \Phi^+$ (resp. a hyperplane $H_{\alpha+k\delta}\in \mathcal A(W,S)$) is {\em small} if there is no hyperplane parallel to $H_{\alpha+k\delta}$ that separates $H_{\alpha+k\delta}$ from $A_\circ$. Denote by $\Sigma:=\Sigma(W,S)$ the set of small roots.  It turns out that  
$$
\Sigma =\Phi_0^+\sqcup (\delta-\Phi_0^+) =  \{\alpha,\delta-\alpha\mid \alpha\in\Phi_0^+\}, 
$$
which is of cardinality $|\Phi_0|$.  It is well-known, see for instance \cite[Proposition 1.4]{Hu90}, that $s\in S_0$ is a bijection on $\Phi_0^+\setminus\{\alpha_s\}$, where $\alpha_s$ denotes the root in $\Delta_0$ such that $s= s_{\alpha_s}$. The next proposition is an analog of this result for $\Sigma$. 

\begin{prop}\label{prop:S0} Let $s\in S_0$ then $s(\Sigma\setminus\{\alpha_s,\delta-\alpha_s\})=\Sigma\setminus\{\alpha_s,\delta-\alpha_s\}$.
\end{prop}
\begin{proof} Let $\alpha\in \Phi^+_0\setminus \{\alpha_s\}$, then $s(\alpha)\in \Phi^+_0\setminus\{\alpha_s\} \subseteq \Sigma\setminus\{\alpha_s,\delta-\alpha_s\}$. We also have by Eq.~(\ref{eq:Ref2}):
$$
s(\delta-\alpha)=\delta-s(\alpha) \in \delta - (\Phi^+_0\setminus\{\alpha_s\})\subseteq \Sigma\setminus\{\alpha_s,\delta-\alpha_s\}.
$$
Therefore $s(\Sigma\setminus\{\alpha_s,\delta-\alpha_s\})\subseteq\Sigma\setminus\{\alpha_s,\delta-\alpha_s\}$; the equality follows since $s$ is injective.
\end{proof}

An element  $w\in W$ is {\em a low element of $(W,S)$} if there is $X\subseteq \Sigma$ such that 
$$
N(w)=\cone_\Phi(X):=\cone(X)\cap\Phi.
$$ 
We denote by $L(W,S)$ the set of low elements of $(W,S)$. If there is no possible confusion, we write $L:=L(W,S)$. Low element were introduced in order to prove that there is a finite Garside shadow, i.e., a finite subset of $W$ containing $S$, closed by taking suffixes and closed under the join operator of the right weak order, see It is proven in \cite[Theorem 1.1]{DyHo16}.

\begin{ex}
We continue Example \ref{ex:A2}. The small roots are:
$$
\Sigma=\{\alpha_1,\alpha_2,\alpha_1+\alpha_2, \delta-\alpha_1,\delta-\alpha_2, \delta-(\alpha_1+\alpha_2)\}.
$$
Obviously $k(e,\alpha_1)=k(e,\alpha_2)=k(e,\alpha_1+\alpha_2)=0$. Now, for an arbitrary $w\in W$, $|k(w,\alpha_1)|$ is given by the number of hyperplanes parallel to $H_{\alpha_1}$  that separates $w\cdot A_\circ$ from $A_\circ$; the sign is then negative if $H_{\alpha_1}$ is one of these hyperplanes (see Remark~\ref{rem:use}). For instance, the $w$ indicated in Figure~\ref{fig:A2} has $k(w,\alpha_1)=1$ since only $H_{\delta-\alpha_1}$ separates the alcove $w\cdot A_\circ$ from $A_\circ$. We proceed similarly for $\alpha_2$ and the highest root $\alpha_1+\alpha_2$. So for the particular $w$ indicated in Figure~\ref{fig:A2} we have:  $k(w,\alpha_2)=0$, since no parallel to $H_{\alpha_2}$ separates $w\cdot A_\circ$ from $A_\circ$;  $k(w,\alpha_1+\alpha_2)=2$, since $H_{\delta-(\alpha_1+\alpha_2)}$ and $H_{2\delta-(\alpha_1+\alpha_2)}$ separates $w\cdot A_\circ$ from $A_\circ$.
\end{ex}

\begin{ex}
We continue Example \ref{ex:B2}. The small roots are:
$$
\Sigma=\{\alpha_1,\alpha_2,\alpha_1+\alpha_2, 2\alpha_1+\alpha_2,\delta-\alpha_1,\delta-\alpha_2, \delta-(\alpha_1+\alpha_2), \delta-(2\alpha_1+\alpha_2)\}.
$$
 Now, for an arbitrary $w\in W$, $|k(w,\beta)|$ is, as in the preceding example, given by the number of hyperplanes parallel to $H_{\beta}$  that separates $w\cdot A_\circ$ from $A_\circ$; the sign is then negative if $H_{\beta}$ is one of these hyperplanes. For instance, the $w$ indicated in Figure~\ref{fig:B2} has $k(w,\alpha_1)=1$ since only $H_{\delta-\alpha_1}$ separates $w\cdot A_\circ$ from $A_\circ$. We proceed similarly for $\alpha_2$, $\alpha_1+\alpha_2$ and the highest root $2\alpha_1+\alpha_2$. So we obtain  $k(w,\alpha_2)=1$, since only $H_{\delta-\alpha_2}$ separates $w\cdot A_\circ$ from $A_\circ$, $k(w,\alpha_1+\alpha_2)=0$, since no parallel to $H_{\alpha_1+\alpha_2}$ separates the alcove $w\cdot A_\circ$ from $A_\circ$, and  finally $k(w,2\alpha_1+\alpha_2)=-1$, since only  $H_{2\alpha_1+\alpha_2}$  separates $w\cdot A_\circ$ from $A_\circ$.
\end{ex}

\subsection{Small inversion sets} In order to check if an element $w\in W$ is a low element, it is often convenient to consider all small roots in the inversion set $N(w)$. 

\smallskip
The {\em small inversion set of $w\in W$} is $\Sigma(w):=N(w)\cap \Sigma$. Therefore, an element $w\in W$ is a low element if and only if $N(w)=\cone_\Phi(\Sigma(w))$.
The set of all small inversion sets is
$
\Lambda(W,S)=\{\Sigma(w)\mid w\in W\} \subseteq 2^\Sigma.
$
We write $\Lambda:=\Lambda(W,S)$ if there is no possible confusion. Notice that, since $\Sigma$ is finite, the set $2^\Sigma$ of subsets of $\Sigma$ is finite, hence $\Lambda$ is finite.

The following result, which defines the transitions in the so-called Brink-Howlett automaton,  is \cite[Lemma 3.21 (1)]{DyHo16};  we refer the reader to~\cite{HoNaWi16,PaYa19,PaYa21} for more information on the Brink-Howlett automaton.

\begin{lem}\label{lem:DH} Let $w\in W$ and $s\in S$ such that $\ell(w)>\ell(sw)$. Then
$
\Sigma(w)=\{\alpha_s\}\sqcup s(\Sigma(sw))\cap \Sigma.
$
\end{lem}

The set of low elements is finite since the map $\sigma:L\to \Lambda$, defined by $w\mapsto \Sigma(w)$, is injective (\cite[Proposition 3.26 (ii)]{DyHo16}). It is conjectured by Dyer and the second author that this map is surjective, which turns out to to be equivalent to Conjecture~\ref{conj:Main} (see \S\ref{ss:ShiSmall} for more details).
 
\begin{conj}[{\cite[Conjecture 2]{DyHo16}}]\label{conj:DH} The map $\sigma:L\to \Lambda$  is a bijection.
\end{conj}

\begin{remark}  The map $\sigma$ is denoted by $\Sigma$ in \cite{DyHo16}, we change the notation to avoid confusion with the set of small roots $\Sigma$.
\end{remark}

\subsection{Basis of inversion sets}  The most efficient way to check if an element $w\in W$ is low is to consider a basis for the inversion set $N(w)$; see \cite[\S4.2]{DyHo16} for details.

Let $w\in W$, consider the set:
\begin{eqnarray*}
N^1(w)&:=&\{\alpha+k\delta\in \Phi^+\mid \mathbb R(\alpha+k\delta) \textrm{ is an extreme ray of }\cone(N(w))\}\\
&=&\{\alpha+k\delta\in \Phi^+\mid \ell(s_{\alpha+k\delta}w)=\ell(w)-1\}.
\end{eqnarray*}
The set $N^1(w)$ is called the {\em basis of the inversion set $N(w)$} and is the inclusion-minimal subset of $N(w)$ such that
$N(w)=\cone_\Phi(N^1(w)).$  

\begin{prop}\label{prop:Basis} Let $w\in W$, then $w\in L$ if and only if $N^1(w)\subseteq \Sigma$. 
\end{prop}
\begin{proof} If $N^1(w)\subseteq \Sigma$, then $w\in L$ by definition of low elements. Assume now that $w\in L$, so 
$
N(w)=\cone_\Phi(\Sigma(w))=\cone_\Phi(N^1(w)).
$
Since $N^1(w)$ is the inclusion-minimal subset with that property, we have  
$
N^1(w)\subseteq \Sigma(w)\subseteq \Sigma.
$
\end{proof}

\subsection{Left and right descent sets} \label{ss:RightLeft}
The basis of inversion sets gives a useful interpretation of the right and left descent sets. The {\em left descent set  of $w\in W$} is the set:
$$
D_L(w)=\{s\in S\mid \ell(sw)=\ell(w)-1\} = \{s\in S\mid \alpha_s\in \Delta\cap N(w)\},
$$
where $\alpha_s\in\Delta$ is the simple root such that $s=s_{\alpha_s}$ in Eq.~(\ref{eq:Ref2}). The {\em set of left descent roots of $w$} is:
$$
ND_L(w)=\Delta\cap N(w)\subseteq N^1(w).
$$
In other words,  $\alpha+k\delta\in ND_L(w)$ if and only if $H_{\alpha+k\delta}$ is a wall of  $A_\circ$ that separates $A_\circ$ from $w\cdot A_\circ$.  Similarly, the {\em right descent set  of $w\in W$} is the set:
$$
D_R(w)=\{s\in S\mid \ell(ws)=\ell(w)-1\} = \{s\in S\mid w(\alpha_s)\in \Phi^-\}.
$$
Let $s\in D_R(w)$, then there is a reduced word for $w$ ending with $s$: $w=us$ with $u\in W$ and $\ell(w)=\ell(u)+1$.  Therefore:
$$
N(u)=N(w)\setminus\{u(\alpha_s)\}. 
$$
The {\em set of right descent roots of $w$}, which is a subset of $N^1(w)$, is
\begin{eqnarray*}
ND_R(w)&=& -w(\{\alpha_s\mid s\in D_R(w)\}) \\
&=& \{\alpha+k\delta\in N(w)\mid \exists u\in W,\ N(u)=N(w)\setminus \{\alpha+k\delta\}\}.
\end{eqnarray*}
In other words, $\alpha+k\delta \in ND_R(w)$ if and only if  $H_{\alpha+k\delta}$  is a wall of $w\cdot A_\circ$ that separates $A_\circ$ from $w\cdot A_\circ$. 

The following proposition gives a useful relationship between the sets of right descent roots of an element $w\in W$ and of one of its maximal suffixes. 

\begin{prop}\label{prop:Desc} We have  $ND_R(sw)= s(ND_R(w)\setminus\{\alpha_s\})$ for all $w\in W$ and $s\in D_L(w)$.
\end{prop}
\begin{proof} Since $s\in D_L(w)$, we have $w=su$ for some $u\in W$ and $\ell(su)=\ell(u)+1$. We first show that $D_R(u) = D_R(w)\setminus(\{w^{-1}sw\}\cap S)$. Let $r\in D_R(u)$, then $u=vr$ for some $v\in W$ and $\ell(u)=\ell(v)+1$. So $w=svr$ with $\ell(w)=\ell(v)+2$. Thus $r\in D_R(w)$. If $r=w^{-1}sw$ then $wr=sw=s(su)=u$. Therefore $w=ur$ and  $r\notin D_R(u)$, a contradiction. So $D_R(u) \subseteq  D_R(w)\setminus(\{w^{-1}sw\}\cap S)$. Now let $r\in  D_R(w)\setminus(\{w^{-1}sw\}\cap S)$. By the exchange condition, we either have $wr=sur=u$ or $r\in D_R(u)$. The first case implies $w^{-1}sw=(u^{-1}s)sw=u^{-1}w=r\in S$, a contradiction. So $r\in D_R(u)$ and we have the desired equality.  Now, by definition of $ND_R(\cdot)$, we have
\begin{eqnarray*}
ND_R(sw)&=&ND_R(u)\\
&=& -u(\{\alpha_r\mid r\in D_R(u)\}) 
= -u(\{\alpha_r\mid r\in D_R(w)\}\setminus(\{w^{-1}(\alpha_s)\}\cap \Delta))\\
&=& -sw(\{\alpha_r\mid r\in D_R(w)\}\setminus(\{w^{-1}(\alpha_s)\}\cap \Delta))\\
&=& s(ND_R(w)\setminus\{\alpha_s\}).
\end{eqnarray*}
\end{proof}

\section{Shi arrangements in affine Weyl groups}\label{se:Shi} Let $(W,S)$ be an affine Coxeter system with underlying Weyl group $W_0$. The {\em Shi arrangement of $(W,S)$} is the affine hyperplane arrangement constituted of the small hyperplanes of $(W,S)$: 
$$
\shi(W,S) = \{H_\beta\mid \beta \in \Sigma\}.
$$
A {\em Shi region of $(W,S)$} is a connected component of the complement of the Shi arrangement in $V_0$, i.e., a connected component of
$$
V_0\setminus \bigcup_{H\in \shi(W,S)} H.
$$

\subsection{Separation sets of Shi regions} Shi regions correspond to equivalence classes of the relation $\sim_\shi$ over $W$ given by $u\sim_\shi v$ if and only if both $u\cdot A_\circ$ and $v\cdot A_\circ$ are contained in the same Shi region.  These equivalence classes have a useful interpretation in terms of separation sets.

\begin{defi}
 The \emph{separation set} of a Shi region $\mathcal R$ is the set of hyperplanes in $\shi(W,S)$ that separate $\mathcal R$ from $A_\circ$. We denote the \emph{inversion set of the Shi region} $\mathcal R$ by:
$$
\Sigma(\mathcal R) = \{\alpha \in \Sigma\mid H_\alpha \textrm{ separates } \mathcal R \textrm{ from } A_\circ\}.
$$
\end{defi}

\begin{prop}\label{prop:ShiSmall} 
Let $w \in W$ and $s \in S.$  Let $\mathcal R$ be a Shi region such that $w\cdot A_\circ \subseteq \mathcal R$.  Then 
\begin{enumerate}
\item  $\Sigma(\mathcal R)=\Sigma(w)$.  In particular $u \sim_\shi v$ if and only if $\Sigma(u) = \Sigma(v)$.
\item If $\alpha_s\in \Sigma(\mathcal R)$ then there exists a Shi region $\mathcal R'$ such that $s\cdot \mathcal R' \subseteq \mathcal R$.
\end{enumerate}
\end{prop}
\begin{proof} 

\noindent \textbf{(1)} Since $w\cdot A_\circ\subseteq \mathcal R$, any hyperplane that separates $\mathcal R$ from $A_\circ$ also separates $w\cdot A_\circ$ from $A_\circ$.  Therefore $\Sigma(\mathcal R)\subseteq N(w)\cap\Sigma=\Sigma(w)$.  Conversely, assume that $H_\alpha$, $\alpha\in\Sigma$, separates $w\cdot A_\circ$ from $A_\circ$. Since $w\cdot A_\circ\subseteq \mathcal R$ and $\mathcal R$ is a connected component of $V_0\setminus \bigcup_{H\in \shi(W,S)} H$,  $H_\alpha$ separates $w\cdot A_\circ$ from $\mathcal R$.

\noindent \textbf{(2)} Let $w\in W$ such that $w\cdot A_\circ\subseteq \mathcal R$. Then, by (1), $\alpha_s\in \Sigma(\mathcal R)=\Sigma(w)\subseteq N(w)$. Therefore  $\ell(sw)<\ell(w)$. Let $\mathcal R'$ be the Shi region such that $sw\cdot A_\circ \subseteq \mathcal R'$. Let $u\in W$ such that $u\sim_\shi sw$, we need to show that $su\sim_\shi w$. Since $\alpha_s\notin N(sw)$, we have $\alpha_s\notin \Sigma(sw)=\Sigma(u)$. In other words, $\ell(su)>\ell(u)$. Then, by Lemma~\ref{lem:DH}, we have $\Sigma(w)=\{\alpha_s\}\sqcup s(\Sigma(sw))\cap \Sigma$,  moreover since $u\sim_\shi sw$ it follows that $\{\alpha_s\}\sqcup s(\Sigma(sw))\cap \Sigma = \{\alpha_s\}\sqcup s(\Sigma(u))\cap \Sigma$.  Then we obtain$ \{\alpha_s\}\sqcup s(\Sigma(u))\cap \Sigma = \{\alpha_s\}\sqcup s(\Sigma(s(su)))\cap \Sigma = \Sigma(su)$ since $\ell(su)>\ell(u).$
Thus $su\sim_\shi w$.
\end{proof}

 \subsection{Minimal elements in Shi regions}\label{ss:LowShi}

In \cite[Proposition 7.1]{Shi88}, Shi shows that any Shi region $\mathcal R$ on an affine Weyl group $W$ contains a unique minimal element, that is,  there is a unique element $w\in W$ with $w \cdot A_\circ \subseteq \mathcal{R}$ and  such that 
$ \ell(w) \leq \ell(u)$, for all $u\sim_\shi w .$
We denote by $L_\shi$ the set of minimal elements in Shi regions.  Shi shows in \cite[Proposition 7.2]{Shi88} that $w\in L_\shi$ if and only if for all $\alpha\in \Phi_0^+$ we have 
\begin{equation}\label{eq:Minstar}
|k(w,\alpha)|=\min\{|k(g,\alpha)|\mid g\sim_\shi w\} .
\end{equation}
In \cite[Proposition 7.3]{Shi88}, Shi characterizes the minimal elements in Shi regions by: 
\begin{equation}\label{eq:star}
  w\in L_\shi \iff w\nsim_\shi ws, \ \forall s\in D_R(w), 
\end{equation}
where $u\nsim_\shi v$ means that  $u$ and $v$ are not in the same Shi region. This characterization implies that any element in a given Shi region is greater or equal in the right weak order than its corresponding minimal element, i.e., the minimal element in a Shi region is the prefix of any other element in that Shi region.


\begin{prop}\label{prop:WeakShi} Let $w\in L_\shi$ and $g\in W$ such that $g\sim_\shi w$, then $N(w)\subseteq N(g)$. 
\end{prop}
\begin{proof} We proceed by induction on $m=\ell(g)-\ell(w)$. Since $w\in L_\shi$ is of minimal length in its Shi region, $m\in \mathbb N$. If $m=0$ then $g=w$ by  unicity of $w\in L_\shi$, so $w=g$ and $N(g)=N(w)$. If $m>0$, then $g\notin L_\shi$ so there is $s\in D_R(g)$ such that $gs\sim_\shi w$ by Eq.~(\ref{eq:star}). Since $\ell(gs)-\ell(w)=\ell(g)-1-\ell(w)=m-1$, we conclude by induction that $N(w)\subseteq N(gs)$. The result follows from the fact that $N(gs) = N(g)\setminus\{-g(\alpha_s)\}$ (see \S\ref{ss:RightLeft}).
\end{proof}
 
 \begin{prop}\label{prop:LLshi} We have $L\subseteq L_\shi$.
 \end{prop}
 \begin{proof}   Let $w\in L$,  then $ND_R(w)\subseteq N^1(w)\subseteq \Sigma$, by Proposition~\ref{prop:Basis}. Let $s\in D_R(w)$, then $-w(\alpha_s)\in ND_R(w)\subseteq \Sigma$ and $N(ws)=N(w)\setminus\{-w(\alpha_s)\}$. Hence $\Sigma(ws)=\Sigma(w)\setminus \{-w(\alpha_s)\}$. Therefore, by Proposition~\ref{prop:ShiSmall},   $w\nsim_\shi ws$. We conclude that $w\in L_\shi$ by Eq.~(\ref{eq:star}). 
  \end{proof}
 
The next proposition shows that $L_\shi$ is closed under taking suffixes

\begin{prop}\label{prop:ShiS} Let $w\in L_\shi$, then $sw \in L_\shi$ for all $s\in D_L(w)$.  Moreover, if $g\sim_\shi sw$ then $sg\sim_\shi w$.  
\end{prop}
\begin{proof} Let $s\in D_L(w)$.  Let $w'\in L_\shi$ be such that $w'\sim_\shi sw$. Then $\Sigma(w')=\Sigma(sw)$. By Lemma~\ref{lem:DH}, since $\ell(w)>\ell(sw)$,  we have 
$
\Sigma(w) = \{\alpha_s\}\cup (\Sigma\cap s(\Sigma(sw)) = \{\alpha_s\}\cup (\Sigma\cap s(\Sigma(w')) =\Sigma(sw').
$
So $w\sim_\shi sw'$ by Proposition \ref{prop:ShiSmall}. But $\ell(sw')=\ell(w')+1\leq \ell(sw)+1= \ell(w)$ implying $sw'=w$ by minimality of $w$.  Hence $sw\in L_\shi$.  Now let $g\sim_\shi sw$, then  $\Sigma(g)=\Sigma(sw)$. Proceeding as above, we obtain
$
\Sigma(w)= \{\alpha_s\}\cup (\Sigma\cap s(\Sigma(sw)) = \{\alpha_s\}\cup (\Sigma\cap s(\Sigma(g))=\Sigma(sg),
$
which means that $sg\sim_\shi w$.
\end{proof}

\subsection{Theorem~\ref{thm:main} implies Conjecture~\ref{conj:Main} and Conjecture~\ref{conj:DH}}\label{ss:ShiSmall}  To prove Theorem~\ref{thm:main}, we only have to prove the converse of Proposition~\ref{prop:LLshi}, that is, to prove that $L_\shi \subseteq L$. Furthermore this statement is also enough to show both conjectures. 

\smallskip
Firstly, as announced in the introduction, Theorem~\ref{thm:main}  implies Conjecture~\ref{conj:Main}.

\begin{prop} Assume that  $L_\shi \subseteq L$, then $L=L_\shi$ and any region $\mathcal R$ of $\shi(W,S)$ contains a unique low element, which is the unique element of minimal length in $\mathcal R$.
\end{prop}
\begin{proof} The first statement follows from Proposition~\ref{prop:LLshi}. Let $\mathcal R$ be a Shi region. Since $L_\shi \subseteq L$, $\mathcal R$ contains a low element $w$, which is of minimal length. The fact that $w$ is the unique low element in $\mathcal R$ follows by definition and Proposition~\ref{prop:ShiSmall}. Assume that $g\in L$ is another low element in $\mathcal R$, then $g\sim_\shi w$. Therefore $\Sigma(w)=\Sigma(g)$ and by definition of low elements:
 $
 N(g)=\cone_\Phi(\Sigma(g))= \cone_\Phi(\Sigma(w))=N(w).
 $
 Therefore $g=w$, since any element of $W$ is uniquely determined by its inversion set. 
\end{proof}

Secondly, to prove Conjecture~\ref{conj:DH} it is enough to show that any Shi region contains a low element, which is implied by the statement $L_\shi \subseteq L$.

\begin{prop}\label{cor:Conj} Assume that any Shi region contains a low element, then the map $\sigma: L \to \Lambda$ is a bijection. 
\end{prop}
\begin{proof} We know already that $\sigma$ is injective by \cite[Proposition 3.26 (ii)]{DyHo16}. Let ${\Sigma(w)\in \Lambda}$ with $w\in W$. Since any Shi region contains a low element, there is $u\in L$ such that $u\sim_\shi w$. Therefore, by Proposition~\ref{prop:ShiSmall}, $\Sigma(u)=\Sigma(w)$. The map $\sigma$ is therefore surjective.
\end{proof}

\subsection{Shi's admissible sign types and Shi regions}\label{se:ShiSigns}\label{ss:Sign} 

In order to show that $L_\shi\subseteq L$ and then to prove Theorem \ref{thm:main} (as explained in \S\ref{ss:ShiSmall}) we need now to survey Shi's admissible sign types.

In \cite{Shi88}, Shi uses the parametrization of the alcoves of $(W,S)$ in order to describe the Shi regions of $\shi(W,S)$.  We follow here Shi's notations from \cite{Shi88}. Let $\overline{\mathscr S}$ be the set of $\Phi_0^+$-tuples over the set $\{-,0,+\}$; its elements are called {\em sign types}. For $w\in W$, the function $\zeta: W\to \overline{\mathscr S}$ is defined as follows:
 $\zeta(w)=(X(w,\alpha))_{\alpha\in\Phi_0^+}$ where
\begin{equation}\label{eq:Xw}
 X(w,\alpha) = 
 \left\{
 \begin{array}{cc}
 + &\textrm{if } k(w,\alpha)>0\\
 0  &\textrm{if } k(w,\alpha)=0\\
 - &\textrm{if } k(w,\alpha)<0 .
 \end{array}
 \right.
\end{equation}

 \begin{defi} \label{def adm sign types} The {\em sign type $X(\mathcal R)=(X(\mathcal R,\alpha))_{\alpha\in\Phi_0^+}$ of a Shi region $\mathcal R$} is  $X(\mathcal R)=\zeta(w)$ for some $w\in W$ such that $w\cdot A_\circ\subseteq \mathcal R$. A sign type $X=(X_\alpha)_{\alpha\in\Phi_0^+}$ is said to be {\em admissible} if $X$ is in the image of $\zeta$.  We denote by $\mathcal{S}_{\Phi}$ the set of admissible sign types of $\Phi$.  
 \end{defi}

The following theorem sums up results from~\cite[Theorem 2.1 and \S6]{Shi88}.

 \begin{thm}[Shi, 1988]\label{thm:Shi} Let $W$ be an irreducible affine Weyl group.
 \begin{enumerate}
 \item  The rank $2$ admissible sign types are precisely those in Figure~\ref{fig:SA2}, Figure~\ref{fig:SB2} and Figure~\ref{fig:SG2}.
 \item A sign type  $X=(X_\alpha)_{\alpha\in\Phi_0^+}$ is admissible if and only if for any irreducible root  subsystem $\Psi$ of rank $2$ in $\Phi_0$, the restriction $(X_\alpha)_{\alpha \in \Psi^+}$ of $X$ to $\Psi^+$ is one of the rank $2$ admissible sign types.  
 \item For $u,v\in W$, we have $u\sim_\shi v$ if and only if $\zeta(u)=\zeta(v)$. 
 \end{enumerate}
 \end{thm}

\begin{figure}[h]
\includegraphics[scale=0.41]{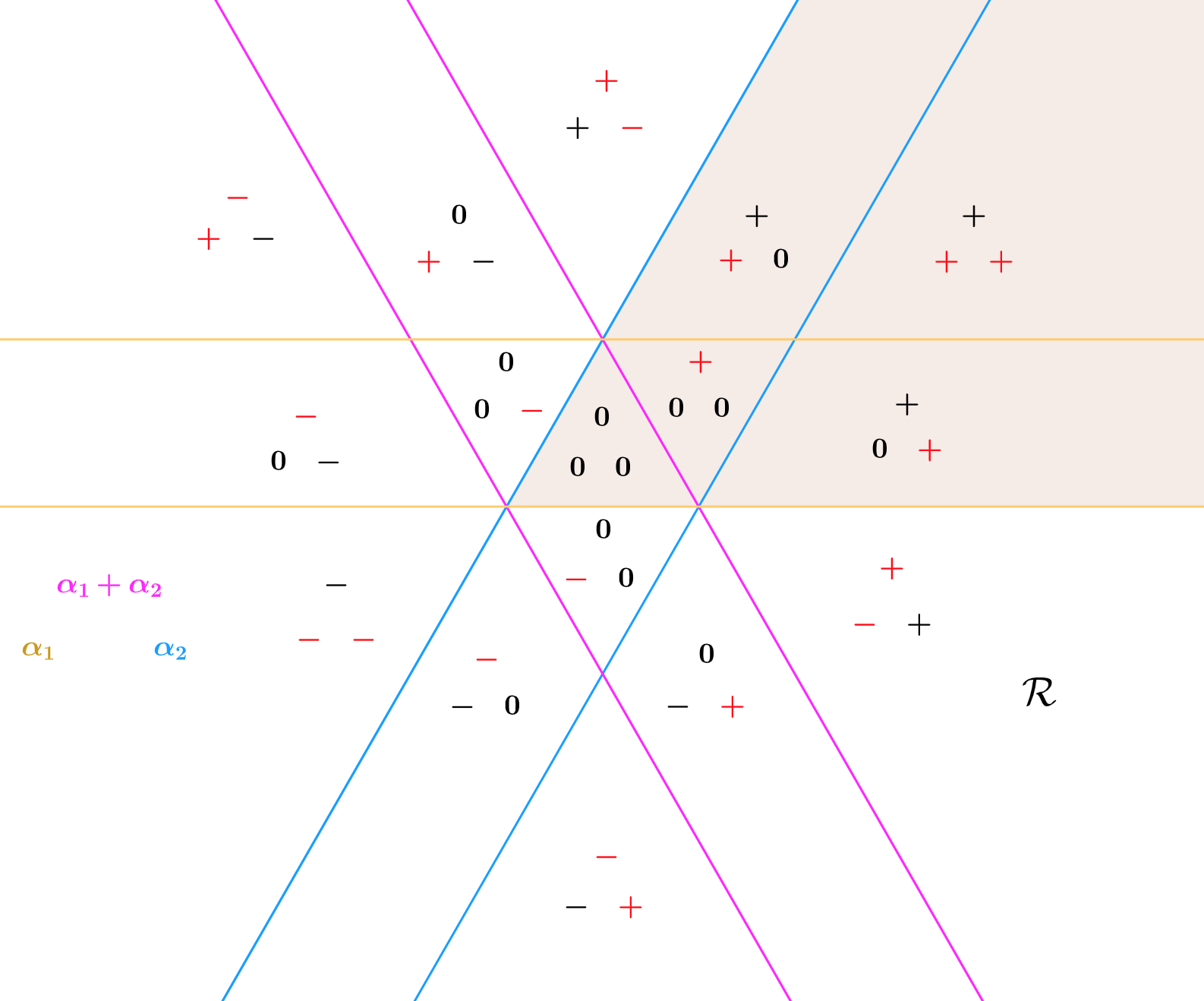}
\caption{The Shi arrangement of type $\tilde A_2$. Each Shi region is labelled with its admissible sign type. 
The labels $X(\mathcal R,\alpha)$ for a Shi region $\mathcal R$ are indicated in each alcove with the parameterization given at the lefthand side of the figure.  The  signs colored in red indicate the descent-roots (see Definition \ref{descent wall/root}) of the corresponding Shi region}
\label{fig:SA2}
\end{figure}

\begin{figure}[h!]
\includegraphics[scale=1.2]{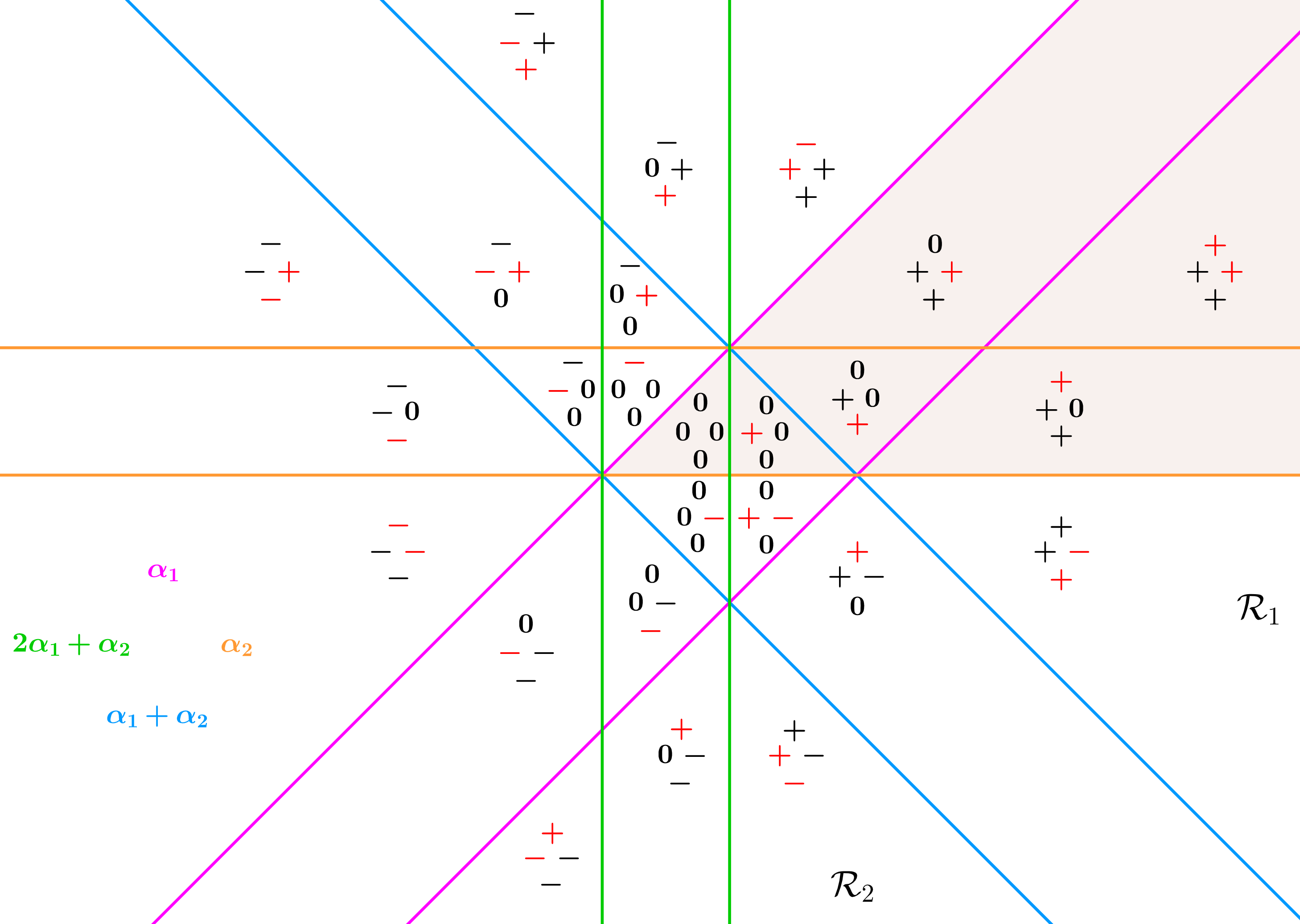}
\caption{The Shi arrangement of type $\tilde B_2$. Each Shi region is labelled with its admissible sign type. 
The labels $X(\mathcal R,\alpha)$ for a Shi region $\mathcal R$ are indicated in each alcove with the parameterization given at the righthand side of the figure. The signs  colored in red indicate the  descent-roots (see Definition \ref{descent wall/root}) of the corresponding Shi region.}
\label{fig:SB2}
\end{figure}

\begin{ex}\label{ex:General} For $(W,S)$ of type $\tilde A$ the irreducible root subsystems of rank $2$ are easy to describe. Set $V=\mathbb{R}^{n+1}$ with the usual orthonormal basis $\{e_1,\dots, e_{n+1}\}$.  We abbreviate $e_i - e_j$ by $e_{ij}:=e_i - e_j$. A way to describe the roots  of $A_n$ is by
$
\Phi=\{\pm (e_{ij}) ~| ~ 1\leq i < j \leq n+1 \}
$
with simple system 
$\Delta = \{ e_{i,i+1} ~| ~ 1\leq i <  n +1\},$
and  positive roots 
$
\Phi^+=\{e_{ij} ~| ~ 1\leq i < j \leq n+1 \}.
$
Thus, the irreducible root subsystems of rank 2 of $\Phi^+$ are of the form 
$
\{e_{ik}, e_{kj}, e_{ij}~|~1\leq i<k<j\leq n+1\}.
$
A convenient way to write the Shi coordinates $k(w,e_{ij})$ of an element $w\in W$ (or the admissible signs $X(\mathcal R,e_{ij})$ of a Shi region $\mathcal R$) is by placing them in a triangular shape as shown on of Figure~\ref{fig:position} for Type $A_4$. 
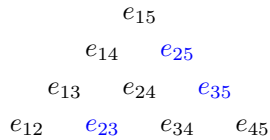
\begin{figure}[h!]
\begin{center}
\begin{tikzpicture} 
\node at (0,0) {$e_{12}$} ;
\node at (1,0) {$\textcolor{blue}{e_{23}}$} ;
\node at (2,0) {$e_{34}$} ;
\node at (3,0) {$e_{45}$} ;
\node at (0.5,0.5 ) {$e_{13}$} ;
\node at (1.5, 0.5 ) {$e_{24}$} ;
\node at (2.5, 0.5) {$\textcolor{blue}{e_{35}}$} ;
\node at (1,1) {$e_{14}$} ;
\node at (2,1) {$\textcolor{blue}{e_{25}}$} ;
\node at (1.5,1.5 ) {$e_{15}$} ;
\end{tikzpicture}
\end{center}
\caption{Triangular presentation of the  positive root system in Type $A_4$. The blue roots show a irreducible root subsystem of rank 2.}
\label{fig:position}
\end{figure}

Then, using this presentation it is easy to check if a sign is admissible using Theorem~\ref{thm:Shi}~{\it (1)}: we only have to look at all the subtriangles  corresponding to  irreducible root subsystems of rank $2$ and check in Figure~\ref{fig:SA2} if the sign type is admissible using  Theorem~\ref{thm:Shi}. An example in rank $5$ is given in Figure~\ref{fig:Gen}:   on the lefthand side, the sign type is admissible, while on the righthand side the sign type is not admissible because the triplet in red does not belong to the rank $2$ admissible sign types given in Figure~\ref{fig:SA2}.

\begin{figure}[h!]
\begin{center}
\begin{tikzpicture} 
\node at (0,0) {$+$} ;
\node at (1,0) {$-$} ;
\node at (2,0) {$+$} ;
\node at (3,0) {$-$} ;
\node at (0.5,0.5 ) {$-$} ;
\node at (1.5, 0.5 ) {$0$} ;
\node at (2.5, 0.5) {$-$} ;
\node at (1,1) {$+$} ;
\node at (2,1) {$-$} ;
\node at (1.5,1.5 ) {$-$} ;

\node at (6,0) {$+$} ;
\node at (7,0) {$\textcolor{red}{-}$} ;
\node at (8,0) {$+$} ;
\node at (9,0) {$-$} ;
\node at (6.5,0.5 ) {$-$} ;
\node at (7.5, 0.5 ) {$0$} ;
\node at (8.5, 0.5) {$\textcolor{red}{-}$} ;
\node at (7,1) {$+$} ;
\node at (8,1) {$\textcolor{red}{+}$} ;
\node at (7.5,1.5 ) {$-$} ;
\end{tikzpicture}
\end{center}
\caption{On the left side the sign type is admissible, while on the right side it is not because the triplet in red does not belong to the rank $2$ admissible sign types given in Figure~\ref{fig:SA2}.}
\label{fig:Gen}
\end{figure}
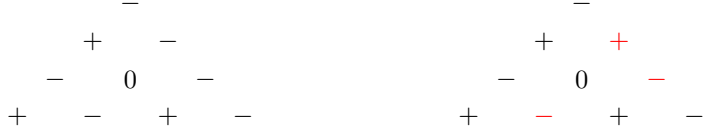
\end{ex}

\begin{remark}
The \say{game} of finding the admissible sign types in types $B,C, D$ is  more complicated because there is no triangular presentation of the positive roots.  An alternative was recently provided by Charles \cite{Charles22}.
\end{remark}

 \begin{prop}\label{prop:Shi} If $\mathcal R$ is a Shi region, then the separation set of $\mathcal R$ is:
 $$
 \Sigma(\mathcal R) = \{\alpha\in \Phi_0^+ \mid X(\mathcal R,\alpha)=-\}\sqcup\{\delta-\alpha\mid \alpha\in \Phi_0^+,\ X(\mathcal R,\alpha)=+\}.
 $$
 \end{prop}
 \begin{proof} Let $w\in W$ such that $w\cdot A_\circ\subseteq \mathcal R$. By Proposition~\ref{prop:ShiSmall}, we have $\Sigma(\mathcal R)=\Sigma(w)\subseteq N(w)$. Let $\alpha\in\Phi_0^+\subseteq \Sigma$. By Lemma \ref{k groupe fini} we have $\alpha\in N(w)\cap \Sigma=\Sigma(w)$ if and only if $k(w,\alpha)<0$. From Eq.~(\ref{eq:Xw}) we get that $\alpha\in \Sigma(w)$ if and only if $X(\mathcal R,\alpha)=-$. The same line of reasoning shows that $\delta-\alpha\in \Sigma(w)$ if and only if $X(\mathcal R,\alpha)=+$, which concludes the proof.
 \end{proof}

 \label{ex:SA2.1}  
  \begin{ex} We continue  Example~\ref{ex:A2}.  Consider the Shi region $\mathcal R$ in Figure~\ref{fig:SA2}. This region is the union of all the alcoves  in Figure~\ref{fig:A2} bounded by the small hyperplanes $H_{\alpha_1}$ and $H_{\delta-(\alpha_1+\alpha_2)}$. This region contains for instance the alcove $w\cdot A_\circ$ such that $k(w,\alpha_1) = -1$, $k(w,\alpha_2)=2$ and $k(w,\alpha_1+\alpha_2)=2$. Therefore $X(\mathcal R,\alpha_1)$ is the sign of $k(w,\alpha_1) = -1$, that is,  $X(\mathcal R,\alpha_1)=-$. Similarly, we obtain $X(\mathcal R,\alpha_2)=+$ and $X(\mathcal R,\alpha_1+\alpha_2)=+$. Finally, we have
  $
   \Sigma(\mathcal R) = \{\alpha_1,\delta-\alpha_2, \delta-(\alpha_1+\alpha_2)\}.
  $
\end{ex}

  \begin{ex}[Type $\tilde B_2$]\label{ex:SB2.1}  We continue Example~\ref{ex:B2}.  Consider the Shi regions $\mathcal R_1$ and $\mathcal R_2$ in Figure~\ref{fig:SB2}. We have
   $
   \Sigma(\mathcal R_1) = \{\alpha_1,\delta-\alpha_2, \delta-(\alpha_1+\alpha_2), \delta-(2\alpha_1+\alpha_2)\}
  $
  together with
   $
 \Sigma(\mathcal R_2) = \{\alpha_1,\delta-\alpha_2, \alpha_1+\alpha_2, \delta-(2\alpha_1+\alpha_2)\}.
  $
\end{ex}

\section{Descent-roots and descent-walls of Shi regions}\label{se:Signs} 

We are now interested in {\em descent-wall} of a Shi region, that is, those walls of a Shi region $\mathcal R$ that separates $\mathcal R$ from the Shi region $A_\circ$. We show at the end of this section the first of the two core results (Lemma~\ref{lem:Key1}) for proving $L_\shi\subseteq L$ and therefore Theorem~\ref{thm:main}. 
\subsection{Walls and descents of Shi regions}   Following classical terminology, a {\em wall of a Shi region $\mathcal R$} is a hyperplane $H\in \shi(W,S)$ such that $H$ supports a facet of~$\mathcal R$.
 
 \begin{defi}[Descent-wall and descent-root] \label{descent wall/root} Let $\mathcal R$ be a Shi region.
 We say that a  wall $H$ of $\mathcal R$ is a {\em descent-wall} of $\mathcal R$ if $H$ separates $\mathcal R$ from $A_\circ$.  In other words, they represent the last hyperplane crossed in a gallery from $A_\circ$ to $\mathcal R$.   The {\em set of descent-roots of a Shi region $\mathcal R$} is
$$
\Sigma D(\mathcal R)=\{\alpha\in \Sigma(\mathcal R)\mid H_{\alpha} \textrm{ is a descent-wall of } \mathcal R\}.
$$
\end{defi}

\begin{defi}
Let $\mathcal{R}$ be a Shi region and let $\alpha \in \Phi_0^+$. We define 
$$
X_{\mathcal{R}}^\alpha :=(X^\alpha_\beta)_{\beta \in \Phi_0^+},$$ 
where $X^\alpha_{\alpha}=0$ and $X^\alpha_\beta=X(\mathcal R,\beta)$ for all $\beta\in \Phi_0^+\setminus\{\alpha\}$. If there is no possible confusion we often omit the $\mathcal{R}$ as subscript of $X_{\mathcal{R}}^\alpha$ and we just write $X^{\alpha}$.
\end{defi}

The following statement, which is an immediate consequence of the definitions,  gives a characterization of the descent-walls of a Shi region in terms of walls, sign types and Shi coefficients:   a wall $H_{\theta}$ is a descent-wall of a Shi region $\mathcal{R}$ if and only if there is a Shi region $\mathcal{R}'$ sharing a common wall with $\mathcal{R}$ such that one sign of $\mathcal R$ changes to a $0$ in the sign type of   $\mathcal{R}'$.  We  point out that the points a) and b) of Proposition \ref{cor:Key1} are of the utmost importance in the proof of out two key results: Lemma~\ref{lem:Key1} and the descent-wall theorem in \S\ref{se:Proof}. Recall also that $\mathcal{S}_{\Phi}$ is the set of Shi admissible sign types. 

\begin{prop}\label{cor:Key1} Let $\mathcal{R}$ be a Shi region with minimal element $w$.  Let $\alpha\in \Phi_0^+$ and let $\theta \in \{\alpha,\delta-\alpha\}$. The following statements are equivalent: 
\begin{enumerate}
\item $H_\theta$ is a descent-wall of $\mathcal R$;
\item there is a Shi region $\mathcal R'$ such that  $\Sigma(\mathcal R')=\Sigma(\mathcal R)\setminus\{\theta\}$;
\item the sign type $X_{\mathcal{R}}^{\alpha}$ is admissible, i.e., $X_{\mathcal{R}}^{\alpha}\in \mathcal{S}_{\Phi}$.
\end{enumerate}
In particular, the set of descent-roots of $\mathcal R$ is:
\begin{eqnarray*}
\Sigma D(\mathcal R)&=&\{\alpha\in \Phi_0^+ \mid X(\mathcal R,\alpha)=-, \ X^{\alpha}_{\mathcal R} \in \mathcal{S}_{\Phi} \} \\
&&\sqcup\ \{\delta-\alpha\mid \alpha\in \Phi_0^+,\ X(\mathcal R,\alpha)=+, \ X^{\alpha}_{\mathcal R} \in \mathcal{S}_{\Phi}\}.
\end{eqnarray*}
\end{prop}

\begin{remark}\label{rem:Descent} In Proposition \ref{cor:Key1},   since $X(\mathcal R)$ is  admissible, it follows from Theorem~\ref{thm:Shi} that the sign type $X^\alpha_{\mathcal R}=(X^\alpha_\beta)_{\beta \in \Phi_0^+}$ is admissible if and only if $(X^\alpha_\beta)_{\beta \in \Psi^+}$ for any irreducible root subsystem $\Psi\subseteq \Phi_0$ of rank $2$ such that $\alpha\in \Psi$.
\end{remark}

\begin{ex}[Type $\tilde A_2$ - continuation of Example~\ref{ex:SA2.1}]\label{ex:SA2.2}  For each admissible sign type of a Shi region in Figure~\ref{fig:SA2}, we colored in red the signs  that correspond to descent-roots in $\Sigma(\mathcal R)$. For instance, for the Shi region $\mathcal R$ in Figure~\ref{fig:SA2}, we have $\Sigma D(\mathcal R) = \{\alpha_1, \delta-(\alpha_1+\alpha_2)\}$, thanks to Proposition~\ref{cor:Key1}. 
\end{ex}

  \begin{ex}[Type $\tilde B_2$ - continuation of Example~\ref{ex:SB2.1}]\label{ex:SB2.2} For each admissible sign type of a Shi region in Figure~\ref{fig:SB2}, we colored in red the signs that corresponds to descent-roots in $\Sigma(\mathcal R)$. For instance, for the Shi regions $\mathcal R_1$ and $\mathcal R_2$ in Figure~\ref{fig:SB2}, we have, thanks to Proposition~\ref{cor:Key1}:
   $$
   \Sigma D(\mathcal R_1) = \{\alpha_2, \delta-(\alpha_1+\alpha_2)\}\text{~}\text{~and~}\text{~}
 \Sigma D(\mathcal R_2) = \{\alpha_1+\alpha_2, \delta-(2\alpha_1+\alpha_2)\}.
  $$
\end{ex}

\medskip

 \begin{prop}\label{prop:LLLshi}   Let $w\in W$, with Shi region $\mathcal R$, then  $w \in L_\shi$ if and only if  $ND_R(w)\subseteq  \Sigma D(\mathcal R)$. In particular the set of descent-walls of a Shi region contains those of its minimal element.
 \end{prop}

 \begin{proof} Assume first that $w \in L_{\text{Shi}}$. Let $\alpha \in ND_R(w)$. We  show that $\alpha \in \Sigma D(\mathcal{R})$.  Since $\alpha \in ND_R(w)$ there is by definition $s \in D_R(w)$ such that $\alpha = -w(\alpha_s)$. Moreover we have $N(ws) = N(w) \setminus \{\alpha\}$ (see \S\ref{ss:RightLeft}). 
  Let $\mathcal{R}'$ be the Shi region containing $ws$. By Proposition~\ref{prop:ShiSmall}~(2), we know that $\Sigma(\mathcal{R}) = \Sigma(w)$ and $\Sigma(\mathcal{R}') = \Sigma(ws)$. Then it follows that:
\begin{align*}
\Sigma(\mathcal{R})\setminus \{\alpha\} & = \Sigma(w) \setminus \{\alpha\}  = (N(w) \cap \Sigma) \setminus \{\alpha\}  = (N(w) \setminus \{\alpha\}) \cap \Sigma \\ &= N(ws) \cap \Sigma  = \Sigma(ws)  = \Sigma(\mathcal{R}').
\end{align*}

Since $w$ is minimal in $\mathcal{R}$ and since $s \in D_R(w)$, it follows by Eq.~(\ref{eq:star}) that $\mathcal{R'} \neq \mathcal{R}$. So $\alpha\in \Sigma$ and, therefore, $\alpha$ is a descent-root of $\mathcal{R}$ by Proposition~\ref{cor:Key1}.
Thus, we have $ND_R(w) \subseteq \Sigma D(\mathcal{R})$. 
The other direction of the statement follows immediately from Eq.~(\ref{eq:star}) and the definitions of descent-roots.
\end{proof}

\subsection{On the descent-roots $\delta-\alpha_s$ for $s\in S_0$} We end this section with a result (Lemma \ref{lem:Key1}), which we need in order to prove Theorem~\ref{thm:main}. We start by discussing the case of rank $2$ affine Weyl groups.

\begin{ex}[On some descent-roots in rank $2$]\label{ex:Rank2} Let $\Phi_0$ be an irreducible  finite crystallographic root system of rank $2$, that is, $\Phi_0$ is of type $A_2$, $B_2$ or $G_2$.  Set $\Delta_0=\{\alpha_1,\alpha_2\}$.  We are particularly interested in the descent-roots corresponding to  $\delta - \alpha_i$, for $i=1,2$. Let $\mathcal R$ be a Shi region of $W$. We know by  Proposition~\ref{prop:Shi} that $\delta-\alpha_i \in \Sigma(\mathcal R)$ if and only if $X(\mathcal R,\alpha_i)=+$. By Proposition~\ref{cor:Key1}, we know that $\delta-\alpha_i \in \Sigma D(\mathcal R)$ if and only if $X(\mathcal R,\alpha_i)=+$ and  the sign type $X^{\alpha_i}$, obtained from $X(\mathcal R)$ by replacing $X(\mathcal R,\alpha_i)=+$ by $0$, is admissible.

\noindent {\bf (Type $A_2$)}: for the Shi region $\mathcal R$ in Figure~\ref{fig:SA2}, $\delta-\alpha_2\notin \Sigma D(\mathcal R)$ since the sign type $X^{\alpha_2}=(X^{\alpha_2}_{\alpha_1},X^{\alpha_2}_{\alpha_1+\alpha_2},X^{\alpha_2}_{\alpha_2})=(-,+,0)$ obtained from $X(\mathcal R)$ by replacing $X(\mathcal R,\alpha_2)=+$ by $0$, is not admissible; recall that all admissible sign types for type $A_2$ are given in Figure~\ref{fig:SA2}.

 More generally, let $\mathcal R$ be  any Shi region in Figure~\ref{fig:SA2} with  $X(\mathcal R,\alpha_2)=+$, i.e.,  $\delta-\alpha_2\in \Sigma(\mathcal R)$ by Proposition~\ref{prop:Shi}. It is easy to check - case by case - that
 $$
 \delta-\alpha_2 \in \Sigma D(\mathcal R) \iff \big(X(\mathcal R,\alpha_1+\alpha_2) = + \implies X(\mathcal R,\alpha_1) \in \{0, +\}\big).
 $$
 $$
 \delta-\alpha_1 \in \Sigma D(\mathcal R) \iff \big(X(\mathcal R,\alpha_1+\alpha_2) = + \implies X(\mathcal R,\alpha_2) \in \{0, +\}\big).
 $$

\noindent {\bf (Type $B_2$)}: for the Shi region $\mathcal R_1$ in Figure~\ref{fig:SB2}, $\delta-\alpha_2\notin \Sigma D(\mathcal R_1)$ since the sign type $X^{\alpha_2}=(X^{\alpha_2}_{\alpha_1},X^{\alpha_2}_{2\alpha_1+\alpha_2},X^{\alpha_2}_{\alpha_1+\alpha_2},X^{\alpha_2}_{\alpha_2})=(0,+,+,-)$ obtained from $X(\mathcal R_1)$ by replacing $X(\mathcal R_1,\alpha_1)=+$ by $0$, is not admissible; recall that all admissible sign types for type $B_2$ are given in Figure~\ref{fig:SB2}. Observe here for example that $X(\mathcal R_1,\alpha_1+\alpha_2) = +$ but $X(\mathcal R_1,\alpha_2) \notin \{0, +\}$.

Now, for the Shi region $\mathcal R_2$ in Figure~\ref{fig:SB2}, $\delta-\alpha_2\notin \Sigma D(\mathcal R_2)$ since the sign type $$X^{\alpha_2}=(X^{\alpha_2}_{\alpha_1},X^{\alpha_2}_{2\alpha_1+\alpha_2},X^{\alpha_2}_{\alpha_1+\alpha_2},X^{\alpha_2}_{\alpha_2})=(0,+,-,-)$$ obtained from $X(\mathcal R_2)$ by replacing $X(\mathcal R_2,\alpha_1)=+$ by $0$ is not admissible. Observe here that $X(\mathcal R_2,\alpha_2+\alpha_1) = +$ but $X(\mathcal R_2,\alpha_2) \notin \{0, +\}$.

More generally, let $\mathcal R$ be  any Shi region in Figure~\ref{fig:SB2} with  $X(\mathcal R,\alpha_2)=+$, i.e.,  $\delta-\alpha_2\in \Sigma(\mathcal R)$. It is easy to check case by case that   $\delta-\alpha_1 \in \Sigma D(\mathcal R)$ if and only if for any $\beta\in \Phi_0^+$ such that $\alpha_1+\beta \in \Phi_0^+$ and  $X(\mathcal R,\beta) = +$ we must have $X(\mathcal R,\beta) \in \{0, +\}$. The same remains valid for $\delta-\alpha_2$. We have therefore in this case the following condition:  $\delta-\alpha_i \in \Sigma D(\mathcal R)$ if and only if:
$$
 \big(\beta \in\Phi_0^+ \textrm{ with }\alpha_i+\beta\in \Phi_0^+ \textrm{ and }X(\mathcal R,\alpha_i+\beta) = + \implies X(\mathcal R,\beta) \in \{0, +\}.
$$

\noindent {\bf (Type $G_2$)}: The above condition remains true in this type for any Shi region $\mathcal R$; we leave the reader check  this assertion on  Figure~\ref{fig:SG2}. For instance, for the region $\mathcal R$ indicated in Figure~\ref{fig:SG2}, $\delta-\alpha_1\notin\Sigma D(\mathcal R)$. Indeed, the sign type $X^{\alpha_1}=(X^{\alpha_1}_{\alpha_1},X^{\alpha_1}_{3\alpha_1+\alpha_2},X^{\alpha_1}_{2\alpha_1+\alpha_2},X^{\alpha_1}_{\alpha_1+\alpha_2},X^{\alpha_2}_{\alpha_2})=(0,+,+,+,-,-)$ obtained from $X(\mathcal R)$ by replacing $X(\mathcal R,\alpha_1)=+$ by $0$, is not admissible. Observe here that $X(\mathcal R,2\alpha_1+\alpha_2) = +$  and $2\alpha_1+\alpha_2 = \alpha_1+(\alpha_1 +\alpha_2)\in\Phi_0^+$ but $X(\mathcal R_1,\alpha_1+\alpha_2) =-\notin \{0, +\}$.
\end{ex}

 \begin{lem}\label{lem:Key1} Let $\mathcal R$ be a  Shi region and $s\in S_0$ such that $\delta-\alpha_s\in \Sigma(\mathcal R)$, i.e., $X(\mathcal R,\alpha_s)=+$. Then $H_{\delta-\alpha_s}$ is a descent-wall of $\mathcal R$ if and only if for any irreducible root  subsystem $\Psi$ of rank $2$ in $\Phi_0$ such that $\alpha_s\in \Psi$ the following condition is verified by $\mathcal R$:
 \begin{itemize}
 \item[($\star$)] If $\beta \in\Psi$ with $\alpha_s+\beta\in \Psi$ and $X(\mathcal R,\alpha_s+\beta) = +$, then $X(\mathcal R,\beta) \in \{0, +\}$.
\end{itemize}
 \end{lem}
 \begin{proof}  First, the case of $\Phi_0$ of rank $2$ is done in Example~\ref{ex:Rank2}.  By Proposition~\ref{cor:Key1}, we know that $H_{\delta-\alpha_s}$ is a descent-wall of $\mathcal R$ if and only if $X^{\alpha_s}$ is admissible. 
 By Theorem~\ref{thm:Shi}, we know that $X^{\alpha_s}$ is admissible if and only if  for any  irreducible root  subsystem $\Psi$  of rank $2$ in $\Phi_0$, the restriction $(X_\gamma^{\alpha_s} )_{\gamma \in \Psi^+}$ of $X^{\alpha_s}$ to $\Psi^+$ is admissible. Since $X(\mathcal R)$ is admissible and that the only difference between $X(\mathcal R)$ and $X^{\alpha_s}$ is that $X(\mathcal R,\alpha_s)\not = 0$ and $X^{\alpha_s}_{\alpha_s}=0$, one has to verify the admissibility of $(X_\gamma^{\alpha_s} )_{\gamma \in \Psi^+}$ only for  irreducible root  subsystems $\Psi$  of rank $2$ in $\Phi_0$ such that $\alpha_s\in \Psi$; it is well-known that $\alpha_s$ is then also a simple root for $\Psi$.
 
 But we know that the lemma is true for $\Psi$ of rank $2$, so $X^{\alpha_s}$ is admissible  if and only if 
 $(X(\mathcal R,\gamma))_{\gamma \in \Psi^+}$ satisfies Condition $(\star)$ for  any  irreducible root  subsystem $\Psi$  rank $2$ in $\Phi_0$ such that $\alpha_s\in \Psi$. 
 \end{proof}

 \section{Dominant Shi regions: the initial case for the proof of Theorem~\ref{thm:main}}\label{se:CelliniPapi} Let $(W,S)$ be an irreducible affine Coxeter system with underlying Weyl group $W_0$.  In \S\ref{ss:ShiSmall}, we explain that it is enough to show that $L_\shi\subseteq L$ in order to prove Theorem~\ref{thm:main}. The proof is by induction with initial case the case of Shi regions contained in the dominant region. In this section, we explain that any dominant Shi region contains a low element, which follows mainly by the works of Cellini and Papi~\cite{CePa00,CePa02,CePa04} on  ad-nilpotent ideals and Shi's article on $\oplus$ sign-types~\cite{Shi97}.

 \subsection{Dominant Shi region}
 Recall from Section \ref{ss:CRS} that the dominant region of  the affine Coxeter arrangement $\mathcal A(W,S)$ is denoted by $C_\circ$.  A Shi region $\mathcal R$ is called {\em dominant} if $\mathcal R\subseteq C_\circ$. An obvious remark is the fact that a Shi region $\mathcal R$ is dominant if and only if $X(\mathcal R,\alpha)\in \{+,0\}$ for all $\alpha\in\Phi_0^+$.
 
 The Weyl group $W_0$ is a standard parabolic subgroup of $W$, since $W_0$ is generated by $S_0\subseteq S$. The set of minimal length coset representatives  is: 
 $$
 {}^0W  := \{v\in W\mid \ell(sv)>\ell(v),  \ \forall s\in S_0\}.
 $$ 
 The following well-known proposition states that an element is a minimal coset representative for $W_0\backslash W$ if and only if its corresponding alcove is in the dominant region. 

 \begin{prop}\label{prop:DomCR} Let $w\in W$. We have $w \in {}^0W$ if and only if $w\cdot A_\circ\subseteq C_\circ$.
 \end{prop}

The inductive step of the proof of Theorem~\ref{thm:main}, which is dealt with in \S\ref{se:Proof}, is based on the decomposition in minimal minimal length coset representative. Let $w\in L_\shi$, then there is $u\in W_0$ and $v\in   {}^0W$ such that $w=uv$ and $\ell(w)=\ell(u)+\ell(v)$. Since $L_\shi$ is closed under taking suffixes (Proposition \ref{prop:ShiS}), then $v\in L_\shi$.  Therefore, if $\ell(u)=0$, i.e., $u=e$, then $w=v\cdot A_\circ\subseteq C_\circ$, which is the initial case of our proof by induction.

  \subsection{Minimal elements in dominant Shi regions} 
 Recall that $\Psi\subseteq \Phi_0^+$ is an ideal  of the root poset  $(\Phi_0^+,\preceq)$ if for all $\alpha\in\Psi$ and $\gamma \in \Phi_0^+$ such that $\alpha\preceq \gamma$ we have $\gamma\in \Psi$. By definition of the root poset, this is equivalent to the condition: $\Psi\subseteq \Phi_0^+$ is an ideal   of $(\Phi_0^+,\preceq)$ if for all $\alpha\in\Psi$ and $\beta \in \Phi_0^+$ such that $\alpha+\beta\in \Phi_0^+$, then $\alpha+\beta\in \Psi$.   
 
 Denote by $\mathcal I(\Phi_0)$ the set of  ideals   of $(\Phi_0^+,\preceq)$ and by $L^0_\shi$ the set of minimal elements in dominant Shi regions.  In \cite[Theorem~1.4]{Shi97}, Shi gave a bijection between dominant Shi regions and $\mathcal I(\Phi_0)$. Hence the number of Shi regions is equal to the number of ideals of the roots poset. We recall here this bijection following the lines of~\cite[\S4]{CePa02}. 
 
It is a general fact that   ideals and antichains  in $(\Phi_0^+,\preceq)$ are in natural bijection, which is given by mapping an ideal $\Psi$ to the set of its minimal elements $\Psi_{\min}$. Let $\Psi\in\mathcal I(\Phi_0)$, then the corresponding Shi region is:
\begin{eqnarray*}
 \mathcal R_\Psi&=&\{x\in C_\circ\mid \langle\beta,x \rangle>1 \textrm{ if } \alpha\preceq \beta \textrm{ for some }\alpha\in \Psi_{\min};\   \langle\beta,x \rangle<1 \textrm{ otherwise }\} .\nonumber
\end{eqnarray*}
In terms of admissible sign types, the bijection  is: 
$$
\Psi\in \mathcal I(\Phi_0) \longmapsto X(\mathcal R_\Psi,\beta) = \left\{
\begin{array}{cc}
+&\textrm{if } \beta\in \Psi\\
0&\textrm{otherwise}
\end{array}
\right. .
$$
Therefore, by Proposition~\ref{prop:Shi}, we have
$\Sigma(\mathcal R_\Psi)=\{ \delta-\beta \mid \beta \in \Psi\}.$ Since $\Psi_{\min}\setminus\{\alpha\}$ is again an antichain for any $\alpha\in\Psi_{\min}$, it corresponds to an ideal $\Psi'$ with $\Psi'_{\min} = \Psi_{\min}\setminus\{\alpha\}$. We deduce (Proposition \ref{prop:WallDom}) that the descent-walls of a dominant Shi region correspond therefore to the roots in  its corresponding antichain.

\begin{prop}\label{prop:WallDom} Let $\Psi \in \mathcal I(\Phi_0)$, then $\Sigma(\mathcal R_\Psi)=\{ \delta-\beta \mid \beta \in \Psi\}$ and 
$$
\Sigma D(\mathcal R_\Psi)= \{\delta-\alpha\mid \alpha\in \Psi_{\inf}\}.
$$ 
 \end{prop}
 
We now describe the element $w_\Psi\in L^0_\shi$ with Shi region $\mathcal R_\Psi$.

 \begin{thm}[Cellini-Papi] \label{thm:CePa} Let $\Psi\subseteq \Phi_0^+$ be an ideal  of the root poset  $(\Phi_0^+,\preceq)$.
 There is a unique element $w_\Psi \in W$ such that 
$
N(w_\Psi) = \cone_\Phi(\{\delta-\alpha\mid\alpha\in \Psi \}).
$
Moreover  $ND_R(w_\Psi)= \{\delta-\alpha\mid \alpha\in \Psi_{\inf}\}$. 
\end{thm}
\begin{proof}  As this result is not stated {\em Per Se} in~\cite{CePa00,CePa04}, we give here a brief proof that point out the relevant results in these articles.  In \cite[\S2]{CePa00}, the authors use the following notations: for  $k\in\mathbb N^*$, write  
$$
\Psi^k:=\{\alpha_1+\dots +\alpha_k\in\Phi_0^+\mid \alpha_i\in \Psi\} \textrm{ and } \mathbb L_\Psi:=\bigcup_{k\in \mathbb N^*} (k\delta-\Psi^k).
$$
Observe that $\mathbb L_\Psi= \cone_\Phi(\{\delta-\alpha\mid\alpha\in \Psi \})$. Now, by  \cite[Theorem 2.6]{CePa00}, there is a unique  $w_\Psi \in W$ such that $\mathbb L_\Psi=N(w_\Psi)$, that is, $N(w_\Psi)=\cone_\Phi(\{\delta-\alpha\mid\alpha\in \Psi \})$.  
\noindent The statement $ND_R(w_\Psi)\supseteq \{\delta-\alpha\mid \alpha\in \Psi_{\inf}\}$ is precisely \cite[Proposition 3.4]{CePa04}, while the reverse inclusion is precisely  \cite[Proposition~2.12~(ii)]{CePa00}.
 \end{proof}

 Cellini and Papi showed in~\cite{CePa02}  that $|L_\shi^0|=|\mathcal I(\Phi_0)|=\cat(\Phi_0)$. We obtain therefore the following corollary.  

\begin{cor}\label{cor:LLshiDom} Let $\Psi\in\mathcal I(\Phi_0)$.
\begin{enumerate}[(i)]
\item The element $w_\Psi$ in Theorem~\ref{thm:CePa} is a low element.
\item The Shi region containing $w_\Psi\cdot A_\circ$ is $\mathcal R_\Psi$ and is therefore dominant.
\item We have $ND_R(w_\Psi)=\Sigma D(\mathcal R_\Psi)= \{\delta-\alpha\mid \alpha\in \Psi_{\inf}\}$.
\item We have $|L^0|=|L_\shi^0|=|\mathcal I(\Phi_0)|= \cat(\Phi_0)$.
\item In particular, $L^0=L_\shi^0=\{w_\Psi\mid \Psi \in \mathcal I(\Phi_0)\}$.
\end{enumerate}
\end{cor}    
\begin{proof} (i) We have $w_\Psi\in L$  by definition of low elements since $\delta-\alpha\in \Sigma$ for all $\alpha\in \Psi\subseteq \Phi_0^+$.

\smallskip
\noindent (ii)  By Theorem~\ref{thm:CePa}, $\Sigma(w_\Psi)=N(w_\Psi)\cap\Sigma =\{ \delta-\beta \mid \beta \in \Psi\}$. Therefore $\Sigma(\mathcal R_\Psi)=\Sigma(w_\Psi)$ by Proposition~\ref{prop:WallDom}.  Hence $w_\Psi\cdot A_\circ\subseteq \mathcal R_\Psi$ By Proposition~\ref{prop:ShiSmall}~(1). Now, by Proposition~\ref{prop:DomCR}, $w_\Psi\cdot A_\circ$ is dominant, which implies that $\mathcal R_\Psi$ is also dominant.

\smallskip
\noindent (iii) follows immediately from Proposition~\ref{prop:WallDom} and Theorem~\ref{thm:CePa}.
\smallskip
\noindent Let us prove (iv):  From (i) and Theorem~\ref{thm:CePa} (1), there is an injective map from $\mathcal I(\Phi_0)$  into $L^0$. By Proposition \ref{prop:LLshi} we have $|\mathcal I(\Phi_0)|\leq |L^0|\leq |L_\shi^0|$. We conclude by Theorem~\ref{thm:CePa}(2) and~\cite{CePa02}.  Then, (v) is a consequence of {(i) and (iv)}.
\end{proof}

 \section{The descent-wall theorem and the inductive case of the proof of Theorem~\ref{thm:main}}\label{se:Proof}  In order to handle the inductive case of the proof of Theorem~\ref{thm:main}, we need to prove first that the descent-walls of an element in $L_\shi$ are precisely the descent-walls of its corresponding Shi region, which takes the most part of this section. Then we conclude the proof of Theorem~\ref{thm:main} in \S\ref{ss:ProofMain}.

\subsection{The descent-wall theorem}

  \begin{thm}\label{thm:ShiLDes} Let  $w\in L_\shi$ with Shi region $\mathcal R$, then $ND_R(w)= \Sigma D(\mathcal R)$. 
\end{thm}

The proof of the so called {\em descent-wall theorem} above is based on Proposition~\ref{prop:Desc} and on its analog  for descent-roots of Shi regions.

\begin{prop}\label{thm:SuffRegion} Let $w\in L_\shi$ with Shi region $\mathcal R$. Let $s\in D_L(w)\cap S_0$ and let $\mathcal R_1$ be the Shi region associated to $sw \in L_\shi$. Then  $\Sigma D(\mathcal R_1)=s(\Sigma D(\mathcal R)\setminus \{\alpha_s\})$. 
\end{prop}

Before proving Proposition~\ref{thm:SuffRegion}, we prove the descent-wall theorem. The proof of Proposition~\ref{thm:SuffRegion} is given in \S\ref{ss:ProofSuffRegion}.

\begin{proof}[Proof of Theorem~\ref{thm:ShiLDes}] We decompose $w=uv$ with $u\in W_0$ and $v \in {}^0W$. We show by induction on $\ell(u)$ that $ND_R(w)= \Sigma D(\mathcal R)$. Notice that $ND_R(w)\subseteq  \Sigma D(\mathcal R)$ by Proposition~\ref{prop:LLLshi}.

If $\ell(u)=0$ then $\mathcal R$ is a dominant Shi region. Therefore $ND_R(w)= \Sigma D(\mathcal R)$ by Proposition~\ref{cor:LLshiDom} $(iii)$. Assume now that $\ell(u)>0$.  Then there is $s\in D_L(u)\subseteq S_0$. Therefore $\alpha_s\in \Sigma(\mathcal R)=\Sigma(w)$ but $\delta-\alpha_s\notin \Sigma(\mathcal R)=\Sigma(w)$.  Set $w'=(su)v$ and $\mathcal R'$ its Shi region. By Proposition~\ref{prop:ShiS} we know that $w'\in L_\shi$ and $\ell(w')=\ell(w)-1$. By induction, we have $ND_R(w')= \Sigma D(\mathcal R')$. By Theorem \ref{thm:SuffRegion} and Proposition~\ref{prop:Desc} we obtain that
$
\Sigma D(\mathcal R)=\{\alpha_s\}\sqcup s(\Sigma D (\mathcal R'))=\{\alpha_s\}\sqcup s(ND_R(w')) = ND_R(w)
$
\end{proof}

 \subsection{Suffix decomposition of some Shi regions} By Proposition~\ref{prop:ShiSmall}, we know that, for a Shi region~$\mathcal R$ such that $\alpha_s\in \Sigma(\mathcal R)$, there is  a Shi region $\mathcal R_1$ such that $s\cdot \mathcal R_1\subseteq \mathcal R$; or equivalently $ \mathcal R_1\subseteq s\cdot \mathcal R$.  In order to prove Proposition~\ref{thm:SuffRegion}, we need first to describe precisely, in term of sign-types, the image $s\cdot \mathcal R$ of Shi region $\mathcal R$ with $s\in S_0$.  
 
\begin{prop}\label{prop:UnionShi}  Let $w\in L_\shi$ with Shi region $\mathcal{R}$, that is, $w\in \widetilde{\mathcal R}$. Let $s \in S_0$ such that $X(\mathcal{R},\alpha_s) = -$; so $\alpha_s\in \Sigma(\mathcal R)$ and $s\in D_L(w)$. Denote $\mathcal R_1$ the Shi region of $sw\in L_\shi$, that is, $sw\in \widetilde{\mathcal R_1}$. Consider the two following subsets of $W$: 
$$
\widetilde{\mathcal{R}_0^s} = \{x \in s \widetilde{\mathcal{R}}~|~k(x,\alpha_s) =0\}\ \textrm{ and } \widetilde{\mathcal{R}_+^s }= \{x \in s \widetilde{ \mathcal{R}}~|~k(x,\alpha_s) >0\}.
$$ 
\begin{enumerate}
\item We have $s\widetilde{\mathcal{R}} = \widetilde{\mathcal{R}_0^s}  \sqcup \widetilde{\mathcal{R}_+^s}$.
\item If $k(w,\alpha_s) = -1$ then $\widetilde{\mathcal{R}_0^s} \neq \emptyset$ is an equivalence class for $\sim_\shi$ that corresponds to the Shi region $\mathcal{R}_0^s =\mathcal R_1$. Moreover $X(\mathcal{R}_1,\alpha_s) = 0$ and $\widetilde{\mathcal{R}_+^s}$ corresponds  to a Shi region $\mathcal{R}_2$ if it is non-empty.
\item If $k(w,\alpha_s) < -1$ then $\widetilde{\mathcal{R}_0^s} = \emptyset$ and  $\widetilde{\mathcal{R}_+^s} \neq \emptyset$ is an equivalence class for $\sim_\shi$ that corresponds to the Shi region $\mathcal{R}_+^s=\mathcal R_1.$ Moreover $\mathcal R_1=s\cdot \mathcal R$ and  $X(\mathcal{R}_1,\alpha_s) = +$. 
\item The signs of the elements of $s\widetilde{\mathcal{R}}$ can only differ over the position $\alpha_s$.  In particular,   $X(\mathcal R_1,\gamma)=X(\mathcal R, s(\gamma))$ for all $\gamma\in \Phi_0^+\setminus\{\alpha_s\}$. 
\end{enumerate}
\end{prop}

See Table~\ref{tab1} or Table~\ref{tab2} for examples in rank $2$.

\begin{proof}  First, observe that $\alpha_s\in \Sigma(w)=\Sigma(\mathcal R)$ by Proposition~\ref{prop:Shi} and Proposition~\ref{prop:ShiSmall}. In particular,  $s\in D_L(w)$ and $sw\in L_\shi$ by Proposition~\ref{prop:ShiS}. 

\smallskip
\noindent  {\em (1)}.  Let $u \in \widetilde{\mathcal{R}}$, then  by  Lemma~\ref{k groupe fini} and Proposition~\ref{descente indice k}  we obtain: 
\begin{equation}\label{eq:prrr}
k(su,\alpha_s) = k(u,s(\alpha_s)) + k(s,\alpha_s) = k(u,-\alpha_s) - 1 = -(k(u,\alpha_s) +1).
\end{equation}
Moreover, since $X(\mathcal{R},\alpha_s) = -$,  we have $k(u,\alpha_s) \leq -1$. If $ k(u,\alpha_s) = -1$ then $k(su,\alpha_s) =0$. Now, if $ k(u,\alpha_s) < -1$, then $k(su,\alpha_s)>0$. Therefore $su \in \widetilde{\mathcal{R}_0^s}  \sqcup \widetilde{\mathcal{R}_+^s}$, which implies that $s \widetilde{ \mathcal{R} }\subseteq  \widetilde{\mathcal{R}_0^s}  \sqcup \widetilde{\mathcal{R}_+^s}$. The converse inclusion is obvious. 

\smallskip
\noindent  {\em (4)}. Let $\gamma \in \Phi_0^+ \setminus \{\alpha_s\}$ and  $u \in \widetilde{\mathcal{R}}$. By Proposition \ref{descente indice k}, we have  $k(su,\gamma) = k(u, s(\gamma))+k(s,\gamma)$. Since $\gamma \in \Phi_0^+ \setminus \{\alpha_s\}$ we have $s(\gamma) \in \Phi_0^+\setminus \{\alpha_s\}$.  Therefore, by Lemma \ref{k groupe fini}, we have $k(s,\gamma)=0$. It follows that $k(su,\gamma) = k(u, s(\gamma))$. Therefore, by Eq.~(\ref{eq:Xw}), we obtain
$X(su,\gamma)=X(u,s(\gamma))=X(\mathcal R,s(\gamma))$, that is, the signs of the elements in $s\widetilde{\mathcal R}$ can only differ over $\alpha_s$.  In particular, for $w\in L_\shi$ and its suffix $sw\in L_\shi$, we obtain:     
$$
 X(\mathcal R_1,\gamma)=X(sw,\gamma)=X(w,s(\gamma))=X(\mathcal R,s(\gamma)),
$$
for all $\gamma \in \Phi_0^+ \setminus \{\alpha_s\}$.

\smallskip
\noindent  {\em (2) and (3)}. By the same arguments as in the proof of {\em (1)} applied to $w\in \widetilde{\mathcal R}$ we obtain  two cases:  $ k(w,\alpha_s) = -1$ and  $ k(w,\alpha_s) < -1$. The details are similar to the cases (1) and (4) above and are left to the reader.
\end{proof}

\subsection{Proof of Proposition~\ref{thm:SuffRegion}}\label{ss:ProofSuffRegion} Let $w\in L_\shi$ with Shi region $\mathcal R$. Let $s\in D_L(w)$ and let $\mathcal R_1$ be the Shi region associated to $sw\in L_\shi$. Notice that $X(\mathcal R,\alpha_s)=-$. 

Thanks to Proposition~\ref{prop:UnionShi}, we have two cases to consider: either $s\cdot \mathcal R_1=\mathcal R$ or there is another Shi region $\mathcal R_2$ distinct from $\mathcal R_1$ such that $s\cdot \mathcal R_2\subseteq \mathcal R$.

The case where $s\cdot \mathcal R_1=\mathcal R$ is settled  by the following lemma.

\begin{lem}\label{cor:ShiUnion} Let $\mathcal R$ be a Shi region and $s\in  S_0$ such that $X(\mathcal R,\alpha_s)=-$. If  a Shi region $\mathcal R_1$ is such that $s\cdot \mathcal R_1=\mathcal R$, then $\Sigma D(\mathcal R_1)=s(\Sigma D(\mathcal R)\setminus \{\alpha_s\})$.
\end{lem}
\begin{proof} Since $s\cdot \mathcal R_1=\mathcal R$,  the walls of $\mathcal R_1$ are the image of the walls of $\mathcal R$ under $s$. So the corollary follows from the following observations: (1) a wall $H$ separates $\mathcal R_1$ from $A_\circ$ if and only if $s\cdot H$ separates $s\cdot \mathcal R_1=\mathcal R$ from $s\cdot A_\circ$; (2) the only wall separating $s\cdot A_\circ$ from $A_\circ$ is $H_{\alpha_s}$ and so $A_\circ$ and $s\cdot A_\circ$ are on the same side of any other hyperplane in the Coxeter arrangement.
\end{proof}

Consider now the case for which there is another Shi region $\mathcal R_2$ distinct from $\mathcal R_1$ such that $s\cdot \mathcal R_2\subseteq \mathcal R$. By Proposition~\ref{prop:UnionShi}~(1) and (4) we have:
$$
X(\mathcal R_1,\alpha_s)=0\textrm{ and } X(\mathcal R_1,\gamma)=X(\mathcal R, s(\gamma))~\textrm{for all  }\gamma\in \Phi_0^+\setminus\{\alpha_s\}.
$$ 
In particular, $\alpha_s\in \Sigma D(\mathcal R)$ but $\alpha_s\notin \Sigma D(\mathcal R_1)$.  

For $\alpha\in \Phi_0^+\setminus \{\alpha_s\}$,  we consider the following sign types:
\begin{itemize}
\item $X^{s(\alpha)}=(X^{s(\alpha)}_\gamma)_{\gamma\in \Phi_0^+}$ is defined by $X^{s(\alpha)}_{s(\alpha)} = 0$ and $X^{s(\alpha)}_\gamma = X(\mathcal R,\gamma)$ for  $\gamma\in \Phi_0^+\setminus\{s(\alpha)\}$. 

\item $Y^\alpha=(Y^\alpha_\gamma)_{\gamma\in \Phi_0^+}$ is defined by $Y^\alpha_\alpha = 0$ and $Y^\alpha_\gamma = X(\mathcal R_1,\gamma)$ for  $\gamma\in \Phi_0^+\setminus\{\alpha\}$. 
\end{itemize}
Thanks to Proposition~\ref{cor:Key1}, we know that $X^\alpha$ is admissible if and only if either $\alpha$ or $\delta-\alpha$ is a descent-root of $\mathcal R$. Similarly, $Y^\alpha$ is admissible if and only if either $\alpha$ or $\delta-\alpha$ is a descent-root of $\mathcal R_1$.

\begin{lem}\label{lem1}  Assume that, for any $\alpha\in \Phi_0^+\setminus\{\alpha_s\}$, we have  $Y^\alpha$ is admissible if and only if $X^{s(\alpha)}$ is admissible. Then $\Sigma D(\mathcal R_1)=s(\Sigma D(\mathcal R)\setminus \{\alpha_s\})$.
\end{lem}
\begin{proof} Let $\alpha\in \Phi_0^+\setminus\{\alpha_s\}$. Notice that $s(\alpha) \in \Phi_0^+\setminus\{\alpha_s\}$. 

We first show that $\Sigma D(\mathcal R_1)\subseteq s(\Sigma D(\mathcal R)\setminus \{\alpha_s\})$. Assume first that $\alpha\in \Sigma D(\mathcal R_1)$, i.e., $X(\mathcal R_1,\alpha)= -$, by Proposition~\ref{cor:Key1}. Then $Y^{\alpha}$ is admissible, so is $X^{s(\alpha)}$ by assumption. Therefore either $s(\alpha)$ or $\delta-s(\alpha)$ is in $\Sigma D(\mathcal R)\setminus\{\alpha_s\}$. Since $X(\mathcal R,s(\alpha))=X(\mathcal R_1,\alpha)= -$, we have $s(\alpha)\in \Sigma D(\mathcal R)\setminus\{\alpha_s\}$, by Proposition~\ref{cor:Key1} again. Therefore $\alpha \in s(\Sigma D(\mathcal R)\setminus \{\alpha_s\})$. If $\delta-\alpha\in \Sigma D(\mathcal R_1)$, then by similar arguments we show that $\delta-\alpha \in s(\Sigma D(\mathcal R)\setminus \{\alpha_s\})$. Hence $\Sigma D(\mathcal R_1)\subseteq s(\Sigma D(\mathcal R)\setminus \{\alpha_s\})$.

Assume now that $\alpha\in s(\Sigma D(\mathcal R)\setminus \{\alpha_s\})$. Therefore  $s(\alpha) \in \Sigma D(\mathcal R)\setminus \{\alpha_s\}$ and $-=X(\mathcal R,s(\alpha))=X(\mathcal R_1,\alpha)$. Hence $X^{s(\alpha)}$ is admissible, so is $Y^\alpha$ by assumption. Then $\alpha\in \Sigma D(\mathcal R_1)$. The case of $\delta-\alpha\in s(\Sigma D(\mathcal R)\setminus \{\alpha_s\})$ is similar. Finally, we get  $\Sigma D(\mathcal R_1)=s(\Sigma D(\mathcal R)\setminus \{\alpha_s\})$.
\end{proof}

Thanks to Lemma~\ref{lem1}, we only need to show now that, for any $\alpha\in \Phi_0^+\setminus\{\alpha_s\}$, we have  $Y^\alpha$ is admissible if and only if $X^{s(\alpha)}$ is admissible.

\begin{lem}\label{lem2} Let $\alpha\in \Phi_0^+\setminus\{\alpha_s\}$. The following statements are equivalent.
\begin{enumerate}[(i)]
\item $Y^\alpha$ is admissible if and only if $X^{s(\alpha)}$ is admissible;

\item for any rank~$2$ irreducible root subsystem $\Psi\subseteq \Phi_0$ such that $\alpha,\alpha_s\in\Psi$, we have  $(Y^\alpha_\gamma)_{\gamma\in \Psi^+}$ is admissible if and only if $(X^{s(\alpha)}_\gamma)_{\gamma\in \Psi^+}$ is admissible.
\end{enumerate}
\end{lem}
\begin{proof} (i) implies (ii) follows directly from Theorem~\ref{thm:Shi}. Assume (ii) to be true for all $\alpha\in \Phi_0^+$.
We describe first four cases that are valid for any rank 2  irreducible rank 2 subsystem $\Psi$ and any $\alpha\in \Phi_0^+\setminus\{\alpha_s\}$. Let $\Psi$ be a rank~$2$ irreducible root subsystem $\Psi\subseteq \Phi_0$ and $\alpha\in \Phi_0^+\setminus\{\alpha_s\}$.  

\medskip

\noindent {\bf Case 1.} If $\alpha,\alpha_s\in\Psi^+$, then $(Y^\alpha_\gamma)_{\gamma\in \Psi^+}$ is admissible if and only if $(X^{s(\alpha)}_\gamma)_{\gamma\in \Psi^+}$ is admissible by assumption.

\medskip

\noindent {\bf Case 2.} If $\alpha_s\in \Psi$ and $\alpha\notin\Psi$.  Then $s(\Psi)=\Psi$ since $\Psi$ a root system containing~$\alpha_s$. Therefore $s(\alpha)\notin \Psi$ since $s(\alpha) =\alpha-2\langle \alpha,\alpha_s^\vee\rangle\alpha_s$. By definition we have for all $\gamma \in \Psi\subseteq \Phi_0^+\setminus\{\alpha,s(\alpha)\}$:
$$
Y^\alpha_\gamma = X(\mathcal R_1,\gamma)\textrm{ and }X^{s(\alpha)}_\gamma = X(\mathcal R,\gamma).
$$
Therefore both $(Y^\alpha_\gamma)_{\gamma\in \Psi^+}$ and $(X^{s(\alpha)}_\gamma)_{\gamma\in \Psi^+}$ are admissible in this case.

\medskip

\noindent {\bf Case 3.} If $\alpha\in \Psi$ and $\alpha_s\notin\Psi$. Then $s(\Psi^+)\subseteq \Phi_0^+\setminus\{\alpha_s\}$ is the positive root system associated to root system $s(\Psi)$ and $s(\alpha)\in s(\Psi^+)$. By definition of $Y^\alpha$, we have $Y^\alpha_\alpha=0$ and for all $\gamma\in \Psi^+\setminus \{\alpha\}$:
$$
Y^\alpha_\gamma=X(\mathcal R_1,\gamma) = X(\mathcal R,s(\gamma)).
$$
Notice that $\gamma \in \Psi^+\setminus \{\alpha\} $ if and only if $s(\gamma)\in s(\Psi^+)\setminus \{s(\alpha)\}$. 
Now, by definition of $X^{s(\alpha)}$, we have $X^{s(\alpha)}_{s(\alpha)}=0$ and for all $\nu=s(\gamma)\in s(\Psi^+)\setminus \{s(\alpha)\}$, or in orther words for all $\gamma \in \Psi^+\setminus \{\alpha\} $, we have: 
$$
X^{s(\alpha)}_\nu=X(\mathcal R,\nu)= X(\mathcal R,s(\gamma))=Y^\alpha_\gamma. 
$$
Therefore the restriction of $Y^\alpha$ to $\Psi$ is admissible if and only if the restriction of $X^{s(\alpha)}$ to $s(\Psi)$ is admissible.

\medskip

\noindent {\bf Case 4.} If $\alpha\notin \Psi$ and $\alpha_s\notin\Psi$. Then $s(\alpha)\notin s(\Psi)$. By the same line of reasoning than in Case 3 we  have for all $\gamma\in \Psi^+\setminus$:
$$
Y^\alpha_\gamma=X(\mathcal R_1,\gamma) = X(\mathcal R,s(\gamma)).
$$
Moreover, for all $\nu=s(\gamma)\in s(\Psi^+)$, or in orther words for all $\gamma \in \Psi^+$, we have: 
$$
X^{s(\alpha)}_\nu=X(\mathcal R,\nu)= X(\mathcal R,s(\gamma))=Y^\alpha_\gamma. 
$$
Therefore the restriction of $Y^\alpha$ to $\Psi$ is admissible if and only if the restriction of $X^{s(\alpha)}$ to $s(\Psi)$ is admissible.

\medskip

\noindent {\bf Conclusion.}  By Theorem~\ref{thm:Shi}, we need to prove that for any rank~$2$ irreducible root subsystem $\Psi\subseteq \Phi_0$,  we have    $(Y^\alpha_\gamma)_{\gamma\in \Psi^+}$ is admissible if and only if $(X^{s(\alpha)}_\gamma)_{\gamma\in \Psi^+}$ is admissible.  Let $\alpha\in \Phi_0^+\setminus\{\alpha_s\}$. 

Assume first that $Y^\alpha$ is admissible. By Theorem~\ref{thm:Shi}, we have to show that the restriction of $X^{s(\alpha)}$ to any rank $2$ irreducible root subsystem $\Psi$ is admissible.  Let $\Psi$ be a rank $2$ irreducible root subsystem. If $\alpha_s\in \Psi$ then by Case 1   the restriction of $X^{s(\alpha)}$ to $\Psi$ is  admissible since $Y^\alpha$ is, and by Case 2,  the restriction of $X^{s(\alpha)}$ to $\Psi$ is  always admissible.  Assume now that $\alpha_s\notin \Psi$.  

If $s(\alpha)\in \Psi$, then we apply Case 3 with $s(\Psi)$ and $\alpha$. Therefore  the restriction of $Y^\alpha$ to $s(\Psi)$ is admissible if and only if the restriction of $X^{s(\alpha)}$ to $s(s(\Psi))=\Psi$ is admissible, which settles that case since $Y^\alpha$ admissible.  If $s(\alpha)\notin \Psi$, then we conclude by Case 4 with $s(\Psi)$ and $\alpha$. So $X^{s(\alpha)}$ is admissible. 

The same line of reasoning using Cases 1 to 4 shows that if  $X^{s(\alpha)}$ then $Y^\alpha$ is admissible. 
\end{proof}

Thanks to the above lemmas, Proposition~\ref{thm:SuffRegion} is a consequence of the following Lemma. 

\begin{lem}\label{ex:Rank2-bis} Let $\Phi_0$ be an irreducible  finite crystallographic root system of rank~$2$, that is, $\Phi_0$ is of type $A_2$, $B_2$ or $G_2$.  Set $\Delta_0=\{\alpha_1,\alpha_2\}$ and $S_0=\{s_1,s_2\}$. Let $i\in\{1,2\}$  and $\mathcal R$ be a Shi region such that $X(\mathcal R,\alpha_i)=-$. 
Let  $\mathcal R_1$, $\mathcal R_2$ be two distinct Shi regions such that $s(\mathcal R_1), s(\mathcal R_2)\subseteq \mathcal R$.  Assume that $X(\mathcal R_1,\alpha_i)=0$. Let $\alpha\in \Phi_0^+\setminus\{\alpha_i\}$, then $Y^{\alpha}$ is admissible if and only if $X^{s_i(\alpha)}$ is admissible.  
\end{lem}
\begin{proof} We have $X(\mathcal R_1,\alpha_i)=0$, $X(\mathcal R_2,\alpha_i)=+$ and $X(\mathcal R_j,\gamma)=X(\mathcal R, s_i(\gamma))$ for all $j=1,2$ and $\gamma\in \Phi_0^+\setminus\{\alpha_i\}$ by Proposition~\ref{prop:UnionShi}~(1). 
  We prove this lemma in type $A_2$ and $B_2$ and leave the case $G_2$ to the reader. The proof works as follow. Let $\alpha\in  \Phi_0^+\setminus\{\alpha_i\}$. We know, thanks to Proposition~\ref{cor:Key1}, how to identify the descent-roots of a Shi region within its sign type; see Examples~\ref{ex:SA2.2} and \ref{ex:SB2.2}. More precisely, in Figure~\ref{fig:SA2} and Figure~\ref{fig:SB2} we have:
  \begin{itemize}
  \item  a sign  $X(\mathcal R_1,\alpha)$  is colored in red if and only if  $Y^\alpha$ is admissible; 
  \item equivalently, a sign  $X(\mathcal R,s_i(\alpha))$  is colored in red if and only if  $X^{s_i(\alpha)}$ is admissible.
  \end{itemize}
  So in order to prove our claim and since $X(\mathcal R_j,\gamma)=X(\mathcal R, s_i(\gamma))$ for all $j=1,2$ and $\gamma\in \Phi_0^+\setminus\{\alpha_i\}$, it is enough in each case to:
\begin{enumerate}
\item list the Shi regions $\mathcal R$  such that $s_i\cdot \mathcal R_1, s_i\cdot \mathcal R_2\subseteq \mathcal R$ for two distinct Shi regions $\mathcal R_1$, $\mathcal R_2$;
\item check that, for all $\beta\in \Phi_0^+\setminus\{\alpha_i\}$, the sign $X(\mathcal R_1,\beta)$ is colored in red if and only if $X(\mathcal R,s_i(\beta))$ is also colored in red.
\end{enumerate}
The result of this case-by-case analysis, which confirms our statement,  is  provided in Table~\ref{tab1} and Table~\ref{tab2} below. Recall that $s_i(\Phi_0^+\setminus\{\alpha_i\})=\Phi_0^+\setminus\{\alpha_i\}$. These tables  are obtained with the help of Figure~\ref{fig:SA2} and Figure~\ref{fig:SB2}, with the  same notations. 
\end{proof}

\begin{center} 
\begin{tabular}{|c|c|cc|cc|}
\hline
 $s_i$  & $X(\mathcal R)$  & $X(\mathcal R_1)$ & $X(\mathcal R_2)$ & $\beta\in \Phi_0^+\setminus\{\alpha_i\}$ with  & $s_i(\beta)\in \Phi_0^+\setminus\{\alpha_i\}$ with\\
  & &  & &  $X(\mathcal R_1,\beta)$  red &$X(\mathcal R,s_i(\beta))$ red \\
  \hline \hline
  $s_1$ & $\begin{array}{c}\textcolor{red}{+}\\ \textcolor{red}{-} \quad +\end{array}$   &$\begin{array}{c}+\\ 0 \quad \textcolor{red}{+}\end{array}$ &$\begin{array}{c}+\\ \textcolor{red}{+} \quad \textcolor{red}{+}\end{array}$& $\alpha_2$& $\alpha_1+\alpha_2$\\
  \hline
  & $\begin{array}{c}0\\ - \quad \textcolor{red}{+}\end{array}$   & $\begin{array}{c}\textcolor{red}{+}\\ 0 \quad 0\end{array}$ & $\begin{array}{c}+\\ \textcolor{red}{+} \quad 0\end{array}$ & $\alpha_1+\alpha_2$& $\alpha_2$\\
  \hline
   & $\begin{array}{c}\textcolor{red}{-}\\ - \quad 0\end{array}$   & $\begin{array}{c}0\\ 0 \quad \textcolor{red}{-}\end{array}$ & $\begin{array}{c}0\\ \textcolor{red}{+} \quad -\end{array}$ & $\alpha_2$& $\alpha_1+\alpha_2$\\
     \hline
   & $\begin{array}{c}-\\ \textcolor{red}{-}\quad \textcolor{red}{-}\end{array}$   & $\begin{array}{c}\textcolor{red}{-}\\ 0 \quad -\end{array}$ & $\begin{array}{c} \textcolor{red}{-}\\ \textcolor{red}{+} \quad -\end{array}$ & $\alpha_1+\alpha_2$& $\alpha_2$\\
  \hline
  \hline
   $s_2$ & $\begin{array}{c}\textcolor{red}{+}\\ + \quad \textcolor{red}{-}\end{array}$   &$\begin{array}{c}+\\  \textcolor{red}{+}\quad 0 \end{array}$ &$\begin{array}{c}+\\ \textcolor{red}{+} \quad \textcolor{red}{+}\end{array}$& $\alpha_1$& $\alpha_1+\alpha_2$\\
  \hline
  & $\begin{array}{c}0\\  \textcolor{red}{+}\quad -\end{array}$  & $\begin{array}{c}\textcolor{red}{+}\\ 0 \quad 0\end{array}$ & $\begin{array}{c}+\\ 0 \quad \textcolor{red}{+} \end{array}$ & $\alpha_1+\alpha_2$& $\alpha_1$ \\
         \hline
   & $\begin{array}{c}\textcolor{red}{-}\\ 0\quad -\end{array}$   & $\begin{array}{c}0\\ \textcolor{red}{-}\quad 0\end{array}$ & $\begin{array}{c} 0\\- \quad  \textcolor{red}{+}\end{array}$ & $\alpha_1$& $\alpha_1+\alpha_2$\\
       \hline
   & $\begin{array}{c}-\\ \textcolor{red}{-}\quad \textcolor{red}{-}\end{array}$   & $\begin{array}{c}\textcolor{red}{-}\\ - \quad 0\end{array}$ & $\begin{array}{c} \textcolor{red}{-}\\- \quad  \textcolor{red}{+}\end{array}$ & $\alpha_1+\alpha_2$& $\alpha_1$\\
  \hline
\end{tabular}
  \captionof{table}{Type $A_2$}\label{tab1}
\end{center}
\medskip
\begin{center}
\begin{tabular}{|c|c|cc|cc|}
\hline
 $s_i$  & $X(\mathcal R)$  & $X(\mathcal R_1)$ & $X(\mathcal R_2)$ & $\beta\in \Phi_0^+\setminus\{\alpha_i\}$ with  & $s_i(\beta)\in \Phi_0^+\setminus\{\alpha_i\}$ with\\
  & &  & &  $X(\mathcal R_1,\beta)$  red &$X(\mathcal R,s_i(\beta))$ red \\
  \hline \hline
  $s_1$ & $\begin{array}{c} - \\ 0 \quad +\\ \textcolor{red}{+} \end{array}$ & $\begin{array}{c}0\\ + \quad 0\\ \textcolor{red}{+}\end{array}$ & $\begin{array}{c}\textcolor{red}{+}\\ + \quad 0\\ +\end{array}$ & $\alpha_1+\alpha_2$ & $\alpha_1+\alpha_2$  \\
  \hline
  & $\begin{array}{c} \textcolor{red}{-} \\ \textcolor{red}{+} \quad +\\ + \end{array}$   & $\begin{array}{c}0\\ + \quad \textcolor{red}{+}\\ +\end{array}$ & $\begin{array}{c}\textcolor{red}{+}\\ +\quad \textcolor{red}{+} \\ +\end{array}$ & $\alpha_2$ & $2\alpha_1+\alpha_2$\\
  \hline
    & $\begin{array}{c} - \\ \textcolor{red}{-} \quad \textcolor{red}{+}\\ 0 \end{array}$   & $\begin{array}{c}0\\ \textcolor{red}{+} \quad \textcolor{red}{-}\\ 0\end{array}$ & $\begin{array}{c}\textcolor{red}{+}\\ +\quad -\\ 0\end{array}$ & $\alpha_2$, $2\alpha_1+\alpha_2$ & $2\alpha_1+\alpha_2$, $\alpha_2$\\
     \hline
    & $\begin{array}{c} - \\ - \quad 0\\ \textcolor{red}{-} \end{array}$   & $\begin{array}{c}0\\ 0 \quad -\\ \textcolor{red}{-}\end{array}$ & $\begin{array}{c}\textcolor{red}{+}\\ 0\quad -\\ -\end{array}$ &  $\alpha_1+\alpha_2$ &  $\alpha_1+\alpha_2$\\
    \hline
        & $\begin{array}{c} \textcolor{red}{-} \\ - \quad \textcolor{red}{-}\\ - \end{array}$   & $\begin{array}{c}0\\ \textcolor{red}{-} \quad -\\ -\end{array}$ & $\begin{array}{c}\textcolor{red}{+}\\ \textcolor{red}{-}\quad -\\ -\end{array}$  &  $2\alpha_1+\alpha_2$&  $\alpha_2$\\
  \hline
  \hline
   $s_2$ & $\begin{array}{c} + \\ + \quad \textcolor{red}{-}\\ \textcolor{red}{+}\end{array}$  & $\begin{array}{c}\textcolor{red}{+}\\ + \quad 0\\ +\end{array}$ & $\begin{array}{c}\textcolor{red}{+}\\ +\quad \textcolor{red}{+}\\ +\end{array}$ & $\alpha_1$ & $\alpha_1+\alpha_2$ \\
  \hline
  & $\begin{array}{c} \textcolor{red}{+} \\ +\quad -\\ 0 \end{array}$   & $\begin{array}{c}0\\ + \quad 0\\ \textcolor{red}{+}\end{array}$ & $\begin{array}{c}0\\ +\quad \textcolor{red}{+} \\ +\end{array}$ & $\alpha_1+\alpha_2$ & $\alpha_1$\\
    \hline
  & $\begin{array}{c}0\\ 0\quad -\\  \textcolor{red}{-}  \end{array}$   & $\begin{array}{c}\textcolor{red}{-} \\ 0 \quad 0\\ 0\end{array}$ & $\begin{array}{c}-\\ 0\quad \textcolor{red}{+} \\ 0\end{array}$ & $\alpha_1$& $\alpha_1+\alpha_2$ \\
  \hline
   & $\begin{array}{c}0\\  \textcolor{red}{-} \quad -\\ - \end{array}$   & $\begin{array}{c}-\\  \textcolor{red}{-}  \quad 0\\ 0\end{array}$ & $\begin{array}{c}-\\  \textcolor{red}{-} \quad \textcolor{red}{+} \\ 0\end{array}$ & $2\alpha_1+\alpha_2$& $2\alpha_1+\alpha_2$ \\ 
     \hline
        & $\begin{array}{c} \textcolor{red}{-} \\ - \quad \textcolor{red}{-}\\ - \end{array}$   & $\begin{array}{c}-\\ - \quad 0\\ \textcolor{red}{-}\end{array}$ & $\begin{array}{c}-\\ -\quad \textcolor{red}{+}\\ \textcolor{red}{-}\end{array}$  &  $\alpha_1+\alpha_2$&  $\alpha_1$\\
  \hline
\end{tabular}
  \captionof{table}{Type $B_2$}\label{tab2}
\end{center}
\smallskip

 \subsection{Conclusion of the proof of Theorem~\ref{thm:main}} \label{ss:ProofMain}  In \S\ref{ss:ShiSmall}, we explain that it is enough to show that $L_\shi\subseteq L$ in order to prove Theorem~\ref{thm:main}. We already know that this statement is true in the dominant region:  $L^0_\shi= L^0$ by Corollary \ref{cor:LLshiDom}.

We first state the following corollary of the descent-wall theorem:  if a hyperplane in $\shi(W,S)$ is parallel to  a hyperplane $H_\alpha$ with $\alpha\in \Delta_0$ and is also in the basis of the inversion set $N^1(w)$ for $w\in L_\shi$, then this hyperplane must be a descent-wall of the alcove of $w$.  

\begin{cor}[of the descent-wall theorem]\label{lem:1} Let $w\in L_\shi$ and $s\in S_0$ such that $\delta-\alpha_s\in N^1(w)$. We have $\delta-\alpha_s \in ND_R(w)$. 
\end{cor}
\begin{proof} Let $\mathcal R$ be the Shi region associated to $w$.  By Theorem~\ref{thm:ShiLDes} we know that $ND_R(w)= \Sigma D(\mathcal R)$.  Assume  that $\delta-\alpha_s \notin  \Sigma D(\mathcal R)$. Then, by Lemma~\ref{lem:Key1}, there is an irreducible root  subsystem $\Psi$ of rank $2$ in $\Phi_0$ such that $\alpha_s\in \Psi$ and there is $\beta \in\Psi$ with $\alpha_s+\beta\in \Psi$, $X(\mathcal R,\alpha_s+\beta) = +$ and $X(\mathcal R,\beta)=-$. So $\delta-(\alpha_s+\beta),\beta \in N(w)$ by Eq.~(\ref{eq:Xw}). Therefore
$
\delta-\alpha_s = \delta-(\alpha_s+\beta)+\beta,
$
which contradicts the fact that  $\delta-\alpha_s\in N^1(w)$ spans  a ray of $\cone(N(w))$. So $\delta-\alpha_s \notin  \Sigma D(\mathcal R)=ND_R(w)$.
\end{proof}

\begin{proof}[Proof of Theorem~\ref{thm:main}] Let $w \in L_\shi$. We decompose $w=uv$ with $u\in W_0$ and $v \in {}^0W$. We show by induction on $\ell(u)$ that $w\in L$. 

Assume first that $\ell(u)=0$, then $u=e$ and $w=v\in L_\shi^0=L^0$ by  Proposition~\ref{cor:LLshiDom}. Now assume that $\ell(u)\geq1$. Then there is $s\in D_L(u)\subseteq D_L(w)\cap S_0$. So $w'=u'v=sw\in L_\shi$ by Proposition~\ref{prop:ShiS}. By induction we have $sw\in L$, since $\ell(u')<\ell(u)$. 

We know that $\alpha_s\notin N(sw)$ so $\alpha_s\notin N^1(sw)$. Assume $\delta-\alpha_s\in N^1(sw)$, then, by Corollary~\ref{lem:1}, $\delta-\alpha_s\in ND_R(sw)$. Since $ND_R(sw)\subseteq s(ND_R(w))$ by Proposition~\ref{prop:Desc}, we have $\delta+\alpha_s=s(\delta-\alpha_s) \in ND_R(w)$. Since $w\in L_\shi$ we obtain by Proposition~\ref{prop:LLLshi}  that $ND_R(w)\subseteq \Sigma$,  which  contradicts $\Sigma=\{\alpha,\delta-\alpha\mid \alpha\in \Phi_0^+\}$. Therefore 
$\delta-\alpha_s\notin N^1(sw)$ and 
$
N^1(sw)\subseteq \Sigma\setminus\{\alpha_s,\delta-\alpha_s\}.
$
By Proposition~\ref{prop:S0}, we obtain therefore that $s(N^1(sw))\subseteq \Sigma\setminus\{\alpha_s,\delta-\alpha_s\}$. We conclude by \cite[Theorem 4.10]{DyHo16} that: 
$$
N^1(w)\subseteq \{\alpha_s\} \cup s(N^1(sw))\subseteq  \{\alpha_s\} \cup(\Sigma\setminus\{\alpha_s,\delta-\alpha_s\})\subseteq \Sigma.
$$
In other words,  $w\in L$.
\end{proof}

\subsection*{Acknowledgment} The authors warmly thank Antoine Abram, Balthazar Charles and  Nathan Williams for numerous interesting discussions on the subject.  The authors are also grateful to Nathan Williams for having pointed us the results in \cite{CePa00,CePa02}. The second author thanks Fran\c cois Bergeron for the references on diagonal invariants.  

Finally we are indebted to the anonymous referees:  their thorough reading of the first version of this article and the comments and suggestions made helped us to considerably improve the content of this article.

\bibliographystyle{plain}

\begin{figure}
\includegraphics[angle=270,scale=0.45]{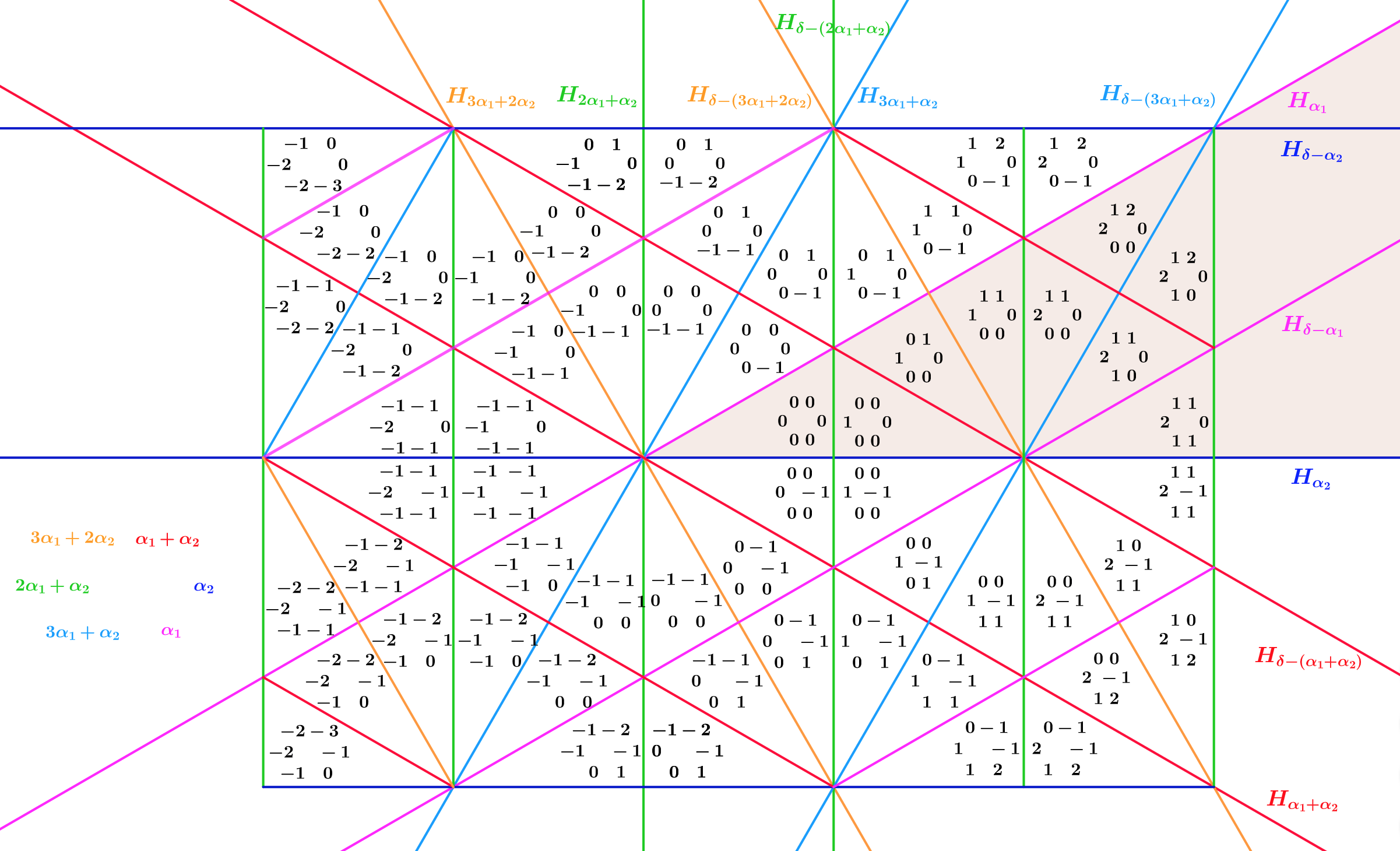}
\caption{{The Coxeter arrangement of type $\tilde G_2$. 
Each alcove is labelled with its Shi parameterization. 
The labels $k(w,\alpha)$ for $\alpha\in \Phi_0^+$ are indicated in each alcove with the parameterization of the finite root system given at the top of the figure. The long roots are 
$
\alpha_2,\ 3\alpha_1+2\alpha_2=s_1(\alpha_2) \textrm{ and } 3\alpha_1+\alpha_2=s_2s_1(\alpha_2).
$ The shaded region is the dominant region of $\mathcal A(W,S)$, which is also the fundamental chamber for the finite Weyl group $W_0$.}}
\label{fig:G2}
\end{figure}

\begin{figure}
\includegraphics[angle=270,scale=1]{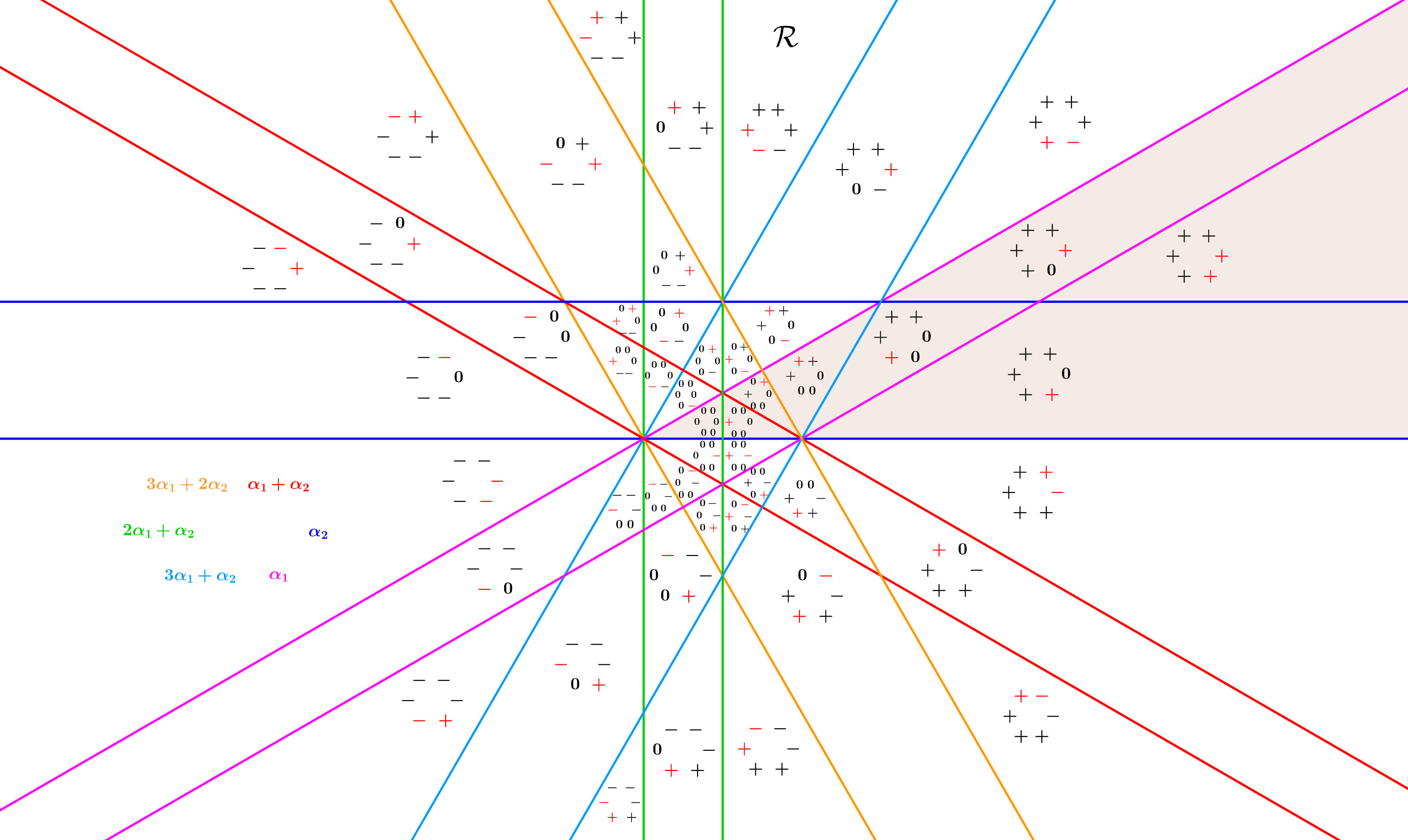}
\caption{The Shi arrangement of type $\tilde G_2$. Each Shi region is labelled with its admissible sign type. 
The labels $X(\mathcal R,\alpha)$ for a Shi region $\mathcal R$ are indicated in each alcove with the parameterization given at the lefthand side of the figure. The shaded region is the dominant region  $C_\circ$ of $\mathcal A(W,S)$. The  signs  colored in red indicate the  descent-roots of the corresponding Shi region.}
\label{fig:SG2}
\end{figure}

\end{document}